\documentclass{birkjour}
\usepackage{hyperref}
\usepackage{color, bm, amscd, tikz-cd}

\renewcommand{\theequation}{\arabic{section}.\arabic{equation}}

\newcommand{\cB}{{\mathcal B}}

\newcommand{\cD}{{\mathcal D}}
\newcommand{\cE}{{\mathcal E}}
\newcommand{\cF}{{\mathcal F}}

\newcommand{\cH}{{\mathcal H}}

\newcommand{\cK}{{\mathcal K}}
\newcommand{\cL}{{\mathcal L}}

\newcommand{\cO}{{\mathcal O}}

\newcommand{\cQ}{{\mathcal Q}}
\newcommand{\cR}{{\mathcal R}}

\newcommand{\cX}{{\mathcal X}}

\newcommand{\bT}{{\mathbb{T}}}
\newcommand{\sbm}[1]{\left[\begin{smallmatrix} #1
		\end{smallmatrix}\right]}

\newcommand{\bD}{{\mathbb D}}

\newcommand{\bDelta}{{\boldsymbol{\Delta}}}
\newtheorem{thm}{Theorem}[section]

\newtheorem{corollary}[thm]{Corollary}
\newtheorem{lemma}[thm]{Lemma}

\newtheorem{proposition}[thm]{Proposition}

\theoremstyle{definition}

\newtheorem{definition}[thm]{Definition}
\newtheorem{remark}[thm]{Remark}

\numberwithin{equation}{section}

\def\textmatrix#1&#2\\#3&#4\\{\bigl({#1 \atop #3}\ {#2 \atop #4}\bigr)}
\def\dispmatrix#1&#2\\#3&#4\\{\left({#1 \atop #3}\ {#2 \atop #4}\right)}
\numberwithin{equation}{section}

\def\textmatrix#1&#2\\#3&#4\\{\bigl({#1 \atop #3}\ {#2 \atop #4}\bigr)}
\def\dispmatrix#1&#2\\#3&#4\\{\left({#1 \atop #3}\ {#2 \atop #4}\right)}

\begin{document}
\title[Commuting Contractions]{Functional Models for Commuting Hilbert-space Contractions}
\author[J. A. Ball]{Joseph A. Ball}
\address{Department of Mathematics, Virginia Tech, Blacksburg, VA 24061-0123, USA.}
\email{joball@math.vt.edu}
\author[H. Sau]{Haripada Sau}
\address{Tata Institute of Fundamental Research, Centre for Applicable Mathematics, Post Bag No 6503, GKVK Post Office,
Sharada Nagar, Chikkabommsandra, Bangalore 560065, India.}
\email{haripadasau215@gmail.com; sau2019@tifrbng.res.in}
\subjclass{Primary: 47A13. Secondary: 47A20, 47A25, 47A56, 47A68, 30H10}
\keywords{Commutative contractive operator-tuples, Functional model, Unitary dilation, Isometric lift, Spectral set, Pseudo-commutative
contractive lift}
\thanks{The research of the second named author was supported by SERB
Indo-US Postdoctoral Research Fellowship, 2017.}

\begin{abstract} We develop a Sz.-Nagy--Foias-type functional model for a commutative contractive operator tuple
$\underline{T} = (T_1, \dots, T_d)$ having  $T = T_1 \cdots T_d$ equal to a completely
nonunitary contraction.  We identify additional
invariants ${\mathbb G}_\sharp, {\mathbb W}_\sharp$ in addition to the Sz.-Nagy--Foias characteristic function
$\Theta_T$ for the
product operator $T$ so that the combined triple $({\mathbb G}_\sharp, {\mathbb W}_\sharp, \Theta_T)$
becomes a complete unitary
invariant for the original operator tuple $\underline{T}$.   For the case $d \ge 3$ in general there is no commutative
isometric lift of
$\underline{T}$;  however there is a (not necessarily commutative) isometric lift having some additional structure
so that, when
compressed to the minimal isometric-lift space for the product operator $T$, generates a special kind of lift of
$\underline{T}$, herein called
a {\em pseudo-commutative contractive lift} of $\underline{T}$, which in turn leads to the functional model
for $\underline{T}$.  This work has many parallels with recently developed model theories for
symmetrized-bidisk contractions
(commutative operator pairs
$(S,P)$ having the symmetrized bidisk $\Gamma$ as a spectral set) and for tetrablock contractions
(commutative operator triples
$(A, B, P)$ having the tetrablock domain ${\mathbb E}$ as a spectral set).
\end{abstract}

\dedicatory{\textit{Dedicated to the memory of Ron Douglas, a leader and dedicated mentor for the field}}

\maketitle

\section{Introduction} A major development in the theory of nonnormal operator theory was the Sz.-Nagy dilation
theorem ({\em any Hilbert-space contraction operator $T$ can be represented as the compression of a unitary operator
to the orthogonal difference of two invariant subspaces}) and the concomitant Sz.-Nagy--Foias functional model for
a completely nonunitary contraction operator (we refer to \cite{Nagy-Foias} for a complete treatment).
Since then there have been many forays into extensions of the formalism
to more general settings.  Perhaps the earliest was that of And\^o \cite{ando} who showed that any pair of
commuting contractions
can be dilated to a pair of commuting unitary operators, but the construction had no functional form like that of
the Sz.-Nagy--Foias model for the single-operator case and did not lead to a functional model for a commutative
contractive pair.  Around the same time the Commutant Lifting Theorem due to Sz.-Nagy-Foias
\cite{Nagy-Foias} appeared, with a seminal special case due to Sarason \cite{Sar}. It was soon realized that there is a
close connection between the And\^o Dilation Theorem and Commutant Lifting (see \cite[Section 3]{Parrott}).
However in the same paper of Parrott it was shown  that And\^o's result fails for $d$ commuting
contractions as soon as $d \ge 3$.   Arveson \cite{Arveson} gave a general operator-algebraic/function-algebraic formulation of the general problem which also revealed the key role of the property of complete contractivity  as opposed to
mere contractivity for representations of operator algebras.

Since the appearance of \cite{BCL}, much work has focused on the $d$-tuple of coordinate multipliers
$M_{z_1}, \dots, M_{z_d}$ on the Hardy space over the polydisk $H^2_{{\mathbb D}^d}$  as well as the
coordinate multipliers  $M_{\zeta_1}, \dots, M_{\zeta_d}$
on the Lebesgue space over the torus $L^2_{{\mathbb T}^d}$ and variations thereof as models for
commuting isometries, and the quest for Wold decompositions related to
variations of these two simple
examples.  While the most definitive results are for the doubly-commuting case
(see \cite{Mandrekar1988, SarkarLAA,  SSW}),
there has been additional progress developing models to handle more general classes of commuting isometries
\cite{SY, Slo1980, BKS, Burdak}.   One can then study examples of commutative contractive tuples by
studying compressions of such commutative
isometric tuples to jointly coinvariant subspaces (see e.g. the book of Douglas-Paulsen \cite{DougPaul} for an
abstract approach and
work of Yang \cite{Yang-survey}).  This work has led to a wealth of distinct new types of examples
with special features, including strong rigidity results (see e.g. \cite{Douglas-Yan}).  In case the commutative
contractive tuple itself is doubly
commuting, one can get a rather complete functional analogue of the Sch\"affer construction of the minimal unitary dilation
(see \cite{Slo1985, BNS-OaM}).

More recent work of Agler and Young  along with collaborators \cite{awy07, AglerYoung00, AglerYoung03},
inspired by earlier work of
Bercovici-Foias-Tannenbaum \cite{bft91} having motivation from the notion of structured singular-value in Robust
Control Theory (see \cite{DP, bft91}), explored more general domains on which to explore the Arveson program:
a broad overview of this direction is given in Section \ref{S:domains} below.  Followup work by Bhattacharyya
and collaborators  (including the second author of the present manuscript)
\cite{Tirtha14, BLS, B-P-SR, BPJOT, BS-PRIMS, BS-CAOT} as well as of Sarkar \cite{SarkarIUMJ} found analogues of
the Sz.-Nagy--Foias defect operator $D_T = (I - T^* T)^{\frac{1}{2}}$ and a more functional form for the dilation and
model theory results established for  these  more general domains (specifically, the symmetrized bidisk $\Gamma$
and tetrablock domain ${\mathbb E}$ to be discussed below).

The goal of the present paper is to adapt these recent advances in the theory of $\Gamma$- and
${\mathbb E}$-function-theoretic operator
theory to the original And\^o-Parrott setting where the domain is the polydisk ${\mathbb D}^d$ and the
associated operator-theoretic object is
a operator-tuple $\underline{T} = (T_1, \dots, T_d)$ of commuting contraction operators on a Hilbert space
$\cH$. Specifically, we adapt the
definition of  {\em Fundamental Operators}, originally introduced in \cite{B-P-SR} for $\Gamma$-contractions and
then adapted to
${\mathbb E}$-contractions in \cite{Tirtha14} , to arrive at a definition of {\em Fundamental Operators}
$\{F_{j1}, F_{j2} \colon j=1, \dots, d\}$ for a commutative contractive operator tuple $\underline{T} = (T_1, \dots, T_d)$.
We then show that the set of Fundamental Operators can be
jointly Halmos-dilated to another geometric object which we call an {\em And\^o tuple} as it appears implicitly as a
key piece in And\^o's
construction of a joint unitary dilation in \cite{ando} for the pair case. While the set of Fundamental Operators is
uniquely determined by
$\underline{T}$, there is some freedom in the choice of And\^o tuple associated with $\underline{T}$.
With the aid of an And\^o tuple,
we are then able to construct a (not necessarily commutative) isometric lift for $\underline{T}$ which has the form of a
Berger-Coburn-Lebow (BCL) model (as in \cite{BCL}) for a commutative isometric operator-tuple.
While any commutative isometric
operator-tuple can be modeled as a BCL-model, there is no tractable characterization as to which BCL-models
are commutative,
except in the $d=2$ case.    For the $d=2$ case it can be shown that there is an appropriate choice of the And\^o tuple
which leads to a commutative BCL-model---thereby giving a more succinct proof of And\^o's original result.
For the general case where $d \ge 3$, we next show how the noncommutative isometric lift constructed from the
And\^o tuple can be cut down to the minimal isometric-lift space for the single product operator $T = T_1 \cdots T_d$
to produce an analogue
of the single-variable lift for the commutative tuple situation which we call a {\em pseudo-commutative contractive lift} of
$\underline{T}$.  For
the case where $T = T_1 \cdots T_d$ is completely nonunitary, we model the minimal isometric-lift space for $T$ as the
 Sz.-Nagy--Foias functional-model space $\sbm{ H^2(\cD_{T^*}) \\ \overline{\Delta_{\Theta_T} L^2(\cD_T)} }$
for $T$ based on the Sz.-Nagy--Foias characteristic function $\Theta_T$ for $T$ and we arrive at a functional model
for the whole commutative
tuple $\underline{T} = (T_1, \dots, T_d)$ consistent with the standard Sz.-Nagy--Foias model for the product operator
$T = T_1 \cdots T_d$.
Let us mention that the basic ingredients of this model already appear in the work of Das-Sarkar-Sarkar \cite{D-S-S}
for the pure-pair case
($d=2$ and $T = T_1 T_2$ has the property that $T^{*n} \to 0$ strongly as $n \to \infty$).
This leads to the identification of additional unitary invariants (in addition to the  characteristic function $\Theta_T$) so
that the whole collection $ \{ {\mathbb G}, {\mathbb W}, \Theta_T\}$
(which we call a {\em characteristic triple} for the commutative contractive tuple $\underline{T}$) is a complete
unitary invariant for $\underline{T}$
for the case where $T = T_1 \cdots T_d$ is completely nonunitary.  Here
${\mathbb G} = \{ G_{j1}, G_{j2}  \colon j=1, \dots, d\}$ consists of the
Fundamental Operators for the adjoint tuple $\underline{T}^* = (T_1^*, \dots, T_d^*)$ and
${\mathbb W} = \{ W_{\sharp 1}, \dots, W_{\sharp d} \}$ consists of a canonically constructed commutative
unitary tuple of multiplication operators on the Sz.-Nagy--Foias
defect model space $\overline{ \Delta_{\Theta_T} \cdot L^2(\cD_{\Theta_T})}$ with product equal to multiplication
by the coordinate
$M_\zeta$ on  $\overline{ \Delta_{\Theta_T} \cdot L^2(\cD_{\Theta_T})}$, all of which is vacuous for the case where
$T = T_1 \cdots T_d$ is pure.  Let us also mention that we obtain an analogue of the Sz.-Nagy--Foias
canonical decomposition for a contraction operator, i.e.:
any commutative contractive operator tuple $\underline{T}$ splits as an orthogonal direct sum
$\underline{T} = \underline{T}_u \oplus \underline{T}_c$ where $\underline{T}_u$ is a commutative unitary
operator-tuple and $\underline{T}_c$ is a commutative contractive operator
tuple with  $T = T_1 \cdots T_d$ completely nonunitary.  As the unitary classification
problem for commutative unitary tuples
can be handled by the spectral theory for commuting normal operators (see \cite{Dix, Arv}), the results for the case where
 $T = T_1 \cdots T_d$ is completely nonunitary combined with the spectral theory for the commutative
unitary case leads to a model theory and unitary classification theory for the general class of commutative
contractive operator-tuples.

Let us mention that Bercovici-Douglas-Foias \cite{BDF1, BDF2, BDF3} have also recently obtained a wealth of
structural information concerning commutative contractive tuples.  This work also builds off the BCL-model
for the commutative isometric case, but also derives additional insight concerning the BCL-model itself.
There also appears the notion of {\em characteristic function} for a commutative contractive operato-tuple,
but this is quite different from our notion of characteristic function (simply the Sz.-Nagy-Foias
characteristic function of the single operator equal to the product $T = T_1 \cdots T_d$).

The paper is organized as follows.   After the present Introduction, Section \ref{S:prelim} on preliminaries provides (i)
a reference for some standard notations to be used throughout, (ii) a review of the rational dilation problem, especially
in the context of  the specific domains $\Gamma$ (symmetrized bidisk) and ${\mathbb E}$ (tetrablock domain),
 including some discussion on how these domains arise from specific examples of the structured singular
 value arising in Robust Control theory,
(iii)  some background on Fundamental Operators in the setting of the symmetrized bidisk, along with  some additional information
(iv) concerning  Berger-Coburn-Lebow models for  commutative isometric-tuples \cite{BCL}  and (v) concerning the Douglas approach
\cite{Doug-Dilation} to the Sz.-Nagy--Foias model theory which will be needed in the sequel.
Let us also mention that the present manuscript is closely related to our companion paper \cite{BS-Memoir} where the results of the present
paper are developed directly for the pair case ($\underline{T} = (T_1, T_2)$ is a commutative contractive operator-pair) from a more
general point of view where additional details are developed.  Finally this manuscript and \cite{BS-Memoir}
subsume the preliminary report
\cite{sauAndo} posted on arXiv.

\smallskip

\textbf{Acknowledgement:}  Finally let us mention that this paper is dedicated to the memory of Ron Douglas, a role
model and inspiring mentor for us.
Indeed it is his approach to the Sz.-Nagy--Foias model theory in \cite{Doug-Dilation} which was a key
intermediate step in our development of the multivariable version appearing here.  In addition his recent work with
Bercovici and Foias \cite{BDF1, BDF2, BDF3} has informed our work as well.

\section{Preliminaries}   \label{S:prelim}

\subsection{Notation}  \label{S:notation}
We here provide a reference for a core of common notation to be used throughout the paper.

Given an operator $A$ on a Hilbert space $\cX$, we write

\smallskip
\noindent
\begin{itemize}

\item $\nu(A)$ = {\em numerical radius} of $A $= $\sup \{ | \langle A x, x \rangle_\cX |
\colon x \in \cX \text{ with } \| x \| = 1\}$.

\smallskip

\item $\rho_{\rm spec}(A)$ = {\em spectral radius} of  $A$ =
$ \sup \{ |\lambda| \colon \lambda \in {\mathbb C} \text{ and } \lambda I - A \text{ not invertible}\}$.

\smallskip

\item If $T \in \cL(\cX)$ with $\| T \| \le 1$, then $D_T$ denotes the {\em defect operator} of $T$ defined as
$D_T = (I - T^* T)^{\frac{1}{2}}$ and $\cD_T = \overline{\operatorname{Ran}}\; D_T$.

\smallskip

\item Given the set of $d$ indices $\{ j \colon 1 \le j \le d\}$, $(j)$ denotes the tuple of $d-1$ indices $(1,\dots,j-1,j+1,\dots,d)$.

\smallskip

\item For a $d$-tuple $(T_1,T_2,\dots,T_d)$ of operators and  an index $j$ such that  $1\leq j \leq d$, $T_{(j)}$ denotes the operator $T_1\cdots T_{j-1}T_{j+1}\cdots T_d$.

\end{itemize}

\subsection{Domains with motivation from control: the symmetrized bidisk $\mathbb{G}$ and the tetrablock ${\mathbb E}$}
\label{S:domains}

The symmetrized bidisk $\mathbb G$ is the domain in ${\mathbb C}^2$ defined as
\begin{equation}   \label{Gamma-def1}
\mathbb G = \{ (s, p ) \in {\mathbb C}^2 \colon \exists\, (\lambda_1, \lambda_2) \in {\mathbb D}^2 \text{ such that }
s = \lambda_1 + \lambda_2 \text{ and } p = \lambda_1 \lambda_2 \}.
\end{equation}
The study of this domain from a function-theoretic and operator-theoretic point of view was initiated in a series of papers by
Agler and Young starting in the late 1990s (see \cite{AglerYoung99, AglerYoungIEOT, AglerYoungTAMS, AglerYoung00, AglerYoung03, AglerYoung04})  with original motivation from Robust Control Theory
(see \cite{DP} and the papers of Bercovici, Foias and Tannenbaum \cite{bft90, bft91, bft95, bft96}). The  control motivation can be explained as follows.

A key role is played by the notion of structured singular value introduced in the control literature by Packer and
Doyle \cite{PD}.  The {\em structured singular value} $\mu_{\bDelta}(A)$ of a $N \times N$ matrix over ${\mathbb C}$
with respect to an {\em uncertainty set} $\bDelta$ (to be thought of as the admissible range for an additional
unknown variable $\Delta$  which is used to parametrize  the set of possible true plants around the chosen nominal
(oversimplified) model plant) is defined to be
$$
 \mu_\bDelta(A) = [ \sup \{ r \in {\mathbb R}_+ \colon I - \Delta A \text{ invertible  for } \Delta \in \bDelta
 \text{ with } \| \Delta | \le r \}]^{-1}
$$
After appropriate normalizations, it suffices to test whether $\mu_\bDelta(A) < 1$;
$$
\mu_\bDelta(A) < 1 \Leftrightarrow I - A \Delta \text{ invertible for all } \Delta \in \bDelta
\text{ with } \| \Delta \| \le 1.
$$
In the control theory context,  this appears as the test for internal stability not only for the nominal plant but for
all other possible true plants as modeled by the uncertainty set $\bDelta$.  In practice the uncertainty set
is taken to be the set of all matrices having a prescribed block diagonal structure.

For the case of $2 \times 2$ matrices, there are three possible block-diagonal structures:
\begin{align*}
& \bDelta_{\rm full} = \left \{ \begin{bmatrix} z_{11}  & z_{12}  \\ z_{21} & z_{22} \end{bmatrix}
 \colon z_{ij} \in {\mathbb C} \right\} =  \text{ all } 2 \times 2 \text{ matrices.}  \\
& \bDelta_{\rm scalar} =\left\{ \begin{bmatrix} z & 0 \\ 0 & z  \end{bmatrix} \colon
z \in {\mathbb C} \right\} =
\text{ all scalar } 2 \times 2 \text{ matrices.}  \\
& \bDelta_{\rm diag} =  \left\{ \begin{bmatrix} z_1 & 0 \\ 0 & z_2 \end{bmatrix} \colon z_1, z_2 \in {\mathbb C}
\right \}  = \text{ all diagonal matrices}.
\end{align*}

An easy exercise using the theory of singular-value decompositions is to show that
$$
  \mu_{\bDelta_{\rm full}}(A) = \| A  \|.
$$

To compute $\mu_{\bDelta_{\rm scalar}}(A)$, one can proceed as follows.
Given $A = \sbm{ a_{11} & a_{12} \\ a_{21} & a_{22} }$, from the definitions we see that
\begin{align}
\mu_{\bDelta_{\rm scalar}}(A) < 1 & \Leftrightarrow \det\left( \begin{bmatrix} 1- z a_{11} & - z a_{12} \\
-z a_{21} &  1 - z a_{22} \end{bmatrix} \right) \ne 0 \text{ for all } z \text{ with } |z| \le 1   \notag \\
& \Leftrightarrow  1 - (\operatorname{tr} A) z  + (\det A) z^2 \ne 0 \text{ for all }  z \text{ with } |z| \le 1.
\label{Gamma-criterion}
\end{align}
Thus the decision as to whether $\mu_{\bDelta_{\rm scalar}}(A) < 1$ depends only on $\operatorname{tr} A$ and
$\det A$, i.e., on $\operatorname{tr} A = \lambda_1 + \lambda_2$ and $\det A = \lambda_1 \lambda_2$ where $\lambda_1$,
$\lambda_2$ are the eigenvalues of $A$. This suggests that we define a map
$\pi_\mathbb G \colon {\mathbb C}^{2 \times 2} \to {\mathbb C}^2$ by
$$
   \pi_\mathbb G(A) = (\operatorname{tr} A, \det A)
$$
and introduce the domain
\begin{align}
\mathbb G' = &  \{ x = (x_1, x_2) \in {\mathbb C}^2 \colon \exists\, A = \sbm{ a_{11} & a_{12} \\ a_{21} & a_{22} }
\in {\mathbb C}^{2 \times 2}  \notag  \\
& \text{ with } \pi_\mathbb G(A) = x \text{ and } \mu_{\bDelta_{\rm scalar}}(A) < 1\}.
 \label{Gamma-def2}
\end{align}
Note next that the first form of the criterion \eqref{Gamma-criterion} for $\mu_{\bDelta_{\rm scalar}}(A) < 1$ can also be
interpreted as saying that  $A$ has no inverse-eigenvalues inside the closed unit disk, i.e., all eigenvalues of $A$
are in the open unit disk, meaning that $\rho_{\rm spec}(A) < 1$.  In this way we see that the symmetrized bidisk
$\mathbb G$ \eqref{Gamma-def1} is exactly the same as
the domain $\mathbb G'$ given by \eqref{Gamma-def2}.  This equivalence gives the connection between the symmetrized
bidisk and the structured singular value $A \mapsto \mu_{\bDelta_{\rm scalar}}(A)$.

Noting that  similarity transformations
$$
  A \mapsto A' = S A S^{-1} \text{ for some invertible } S
$$
preserve eigenvalues and using the fact that $\rho_{\rm spec}(A) < 1$ if and only if $A$ is similar to a strict
contraction (known as Rota's Theorem \cite{Rota} among mathematicians whereas engineers think in terms of
$X = S^*S \succ 0$ being a solution of the Linear Matrix Inequality $A^* X A - X \prec 0$---see e.g.\ \cite[Theorem 11.1 (i)]{DP}),
we see that yet another characterization of the domain $\mathbb G$ is
\begin{equation}   \label{Gamma-def3}
\mathbb G = \{ x = (s, p ) \in {\mathbb C}^2 \colon  \exists\,  A \in {\mathbb C}^2 \text{ with } \pi_\mathbb G(A) = x \text{ and }
\| A \| < 1\}.
\end{equation}
The fact that one can always write down a companion matrix $A$ whose characteristic polynomial $\det( z I - A)$
is equal to a given polynomial $1 - s z + p s^2$ leads us to one more equivalent definition of $\mathbb G$:
\begin{equation} \label{Gamma-def4}
 \mathbb G = \{ (s, p) \in {\mathbb C}^2 \colon 1 - s z + p z^2 \ne 0 \text{ for } |z| \le 1 \}.
 \end{equation}The closure of $\mathbb{G}$ is denoted by $\Gamma$.

A similar story holds for the tetrablock domain ${\mathbb E}$ defined as
\begin{equation}   \label{E-def1}
{\mathbb E} = \{ x  = (x_1, x_2, x_3) \in {\mathbb C}^3 \colon 1 - x_1 z - x_2 w + x_3 zw \ne 0
\text{ whenever } |z| \le 1, |w| \le 1 \}
\end{equation}
(the analogue of definition \eqref{Gamma-def4} for the symmetrized bidisk $\mathbb G$)
and its connection with the structured singular value $A \mapsto \mu_{\bDelta_{\rm diag}}(A)$.
From  the definitions we see that, for $A = \sbm{ a_{11}
& a_{12} \\ a_{21} & a_{22}}$,
\begin{align*}
\mu_{\bDelta_{\rm diag}}(A) < 1 & \Leftrightarrow
 \det \left( \begin{bmatrix} 1 & 0 \\ 0 & 1 \end{bmatrix} - \begin{bmatrix} z & 0 \\ 0 & w\end{bmatrix}
  \begin{bmatrix} a_{11} & a_{12} \\ a_{21} & a_{22} \end{bmatrix} \right) \ne 0 \text{ for } |z| \le 1,\, |w| \le 1 \\
  & \Leftrightarrow
 1 - z a_{11} - w a_{22} + zw \cdot \det A  \ne 0 \text{ whenever } |z| \le 1,\, |w| \le 1.
\end{align*}
This suggests that we define a mapping $\pi_{\mathbb E}\colon {\mathbb C}^{2 \times 2} \to {\mathbb C}^3$ by
$$
  \pi_{\mathbb E} \left( \sbm{ a_{11} & a_{12} \\ a_{21} & a_{22} } \right) = (a_{11}, a_{22},  a_{11} a_{22} - a_{12} a_{21} )
$$
and we define a domain ${\mathbb E}$ by
\begin{equation}  \label{E-def2}
{\mathbb E} = \{ x = (x_1, x_2, x_3) \in {\mathbb C}^3 \colon \exists\, A \in {\mathbb C}^{2 \times 2} \text{ with }
\pi_{\mathbb E}(A) = x \text{ and } \mu_{\bDelta_{\rm diag}}(A) < 1\}.
\end{equation}
If $x = (x_1, x_2, x_3)$ belongs to  ${\mathbb E}$ as defined in \eqref{E-def1} above, we can always take
$A = \sbm{ x_1 & x_1 x_2 - x_3 \\ 1 & x_2}$ to produce a matrix $A$ with $\pi_{\mathbb E}(A) = (x_1, x_2, x_3)$ and then
this $A$ has the property that $\mu_{\bDelta_{\rm diag}}(A) < 1$.  Thus definitions \eqref{E-def1} and \eqref{E-def2}
are equivalent.  Among the many equivalent definitions of ${\mathbb E}$ (see \cite[Theorem 2.2]{awy07}), one of the
more remarkable ones is the following variation of definition \eqref{E-def2}:
\begin{equation}  \label{E-def3}
{\mathbb E} = \{ x = (x_1, x_2, x_3) \in {\mathbb C}^3 \colon \exists\, A \in {\mathbb C}^3 \text{ with } \pi_{\mathbb E}(A) = x
\text{ and } \| A \| < 1\}.
\end{equation}
That \eqref{E-def2} and \eqref{E-def3} are equivalent can be seen as a consequence of the $2s+f$ theorem in the
control literature (with $s=0$, $f=2$ so that
$2s+f = 2 \le 3$)---see \cite[Theorem 8.27]{DP}, but is also proved in \cite{awy07} directly.

While the original motivation was the control theory connections, most of the ensuing research concerning the domains
${\mathbb G}$ and ${\mathbb E}$ focused on their role as new concrete domains to explore operator- and function-theoretic
questions concerning general domains in ${\mathbb C}^d$. One such question is  the {\em rational dilation problem}
formulated by Arveson \cite{Arveson}.  Let us assume that $K$ is a compact set in $\mathbb{C}^d$ (as is the case for
$K$ equal to $\Gamma=\overline{\mathbb{G}}$ or $\overline{\mathbb E}$).  Suppose that we are given a commutative tuple $\underline{T}
= (T_1, \dots, T_d)$ of Hilbert space operators with Taylor joint spectrum contained in $K$ (if the Hilbert space $\cH$ is
finite-dimensional,  one can take Taylor joint spectrum to mean the set of joint eigenvalues).  If $r$ is any function
holomorphic in a neighborhood of $K$ (taking $r$ to be rational suffices: if $K$ is polynomially convex, one can even take $r$ to be
polynomial) any reasonable functional calculus can be used to define $r(\underline{T})$.  We say that {\em $\underline{T}$
is a $K$-contraction} (sometimes also phrased as {\em $K$ is a spectral set for $\underline{T}$}), if for all $r \in \operatorname{Rat}(K)$
(rational functions holomorphic in a neighborhood of $K$) it is the case that the following {\em von Neumann inequality} holds:
$$
   \|r(\underline{T}) \|_{\cB(\cH)} \le \| r \|_{\infty, K} = \sup_{z \in K} \{|r(z)|\}
$$
where $\cB(\cH)$ is the Banach algebra of  bounded linear operators on $\cH$ with the operator norm. Let us say that that
operator tuple $\underline{U} = (U_1, \dots, U_d)$ is {\em $K$-unitary} if $\underline{U}$ is a commutative tuple of
normal operators with joint spectrum contained in the distinguished boundary $\partial_e K$ of $K$.  We say that
$\underline{T}$ has a {\em $K$-unitary dilation}  if there is a $K$-unitary operator-tuple $\underline{U}$ on a
larger Hilbert space $\cK$ containing $\cH$ such that $r(\underline{T}) = P_\cH r(\underline{U})|_\cH$ for all $r \in \operatorname{Rat} K$. If
$\underline{T}$ has a $K$-unitary dilation $\underline{U}$, it follows that
\begin{align*}
\| r(\underline{T}) \| & = \| P_\cH r(\underline{U})|_\cH \| \le \| r(\underline{U}) \|
 = \sup_{z \in \partial_e K} |r(z) |  \\
 & \text{(by the functional calculus for commutative  normal operators)} \\
& = \sup_{z \in K} |r(z)| \text{ (by the definition of the distinguished boundary)}
\end{align*}
and it follows that $\underline{T}$ has $K$ as a spectral set.
The {\em rational dilation question} asks: for a given
compact set $K$, when is it the case that the converse direction holds, i.e., that  $\underline{T}$ being a $K$-contraction
implies that $\underline{T}$ has a $K$-unitary dilation $\underline{U}$?
For the case of $K$ equal to the closed polydisk ${\overline{\mathbb D}}^d$, the rational dilation question is known to
have an affirmative answer in case $d=1$ (by the Sz.-Nagy dilation theorem \cite{sz-nagy}) as well as $d=2$
(by the And\^o dilation theorem \cite{ando}) but has a negative answer for $d\ge 3$ by the result of Parrott \cite{Parrott}.
For the case of $K = \Gamma$ it is known that the rational dilation question has an affirmative answer (\cite{AglerYoung03,
B-P-SR})  while the case of  $K = {\overline{\mathbb E}}$ was initially thought to be settled in the negative \cite{Pal}
but now appears to be still undecided \cite{BS-failure}.

It is known that existence of a $K$-unitary dilation
for $\underline{T}$ is equivalent to the existence of a {\em $K$-isometric lift} for $\underline{T}$.  Here a commutative
operator-tuple $\underline{V} = (V_1, \dots, V_d)$ defined on a Hilbert space $\cK_+$ is said to be a
{\em $K$-isometry} if there is a $K$-unitary $d$-tuple $\underline{U} = (U_1, \dots, U_d)$ on a Hilbert space $\cK$
containing $\cK_+$ such that $\cK_+$ is invariant for $\underline{U}$  and $\underline{U}$ restricted to $\cK_+$ is
equal to $\underline{V}$, i.e.,
$$
U_j \cK_+ \subset \cK_+ \text{ and }
  U_j|_{\cK_+} = V_j \text{ for } j = 1, \dots, d.
$$
Then we say that the $K$-contraction $\underline{T}$ has a $K$-isometric lift if there is a
$K$-isometric operator-tuple $\underline{V}$ on a Hilbert space $\cK_+$ containing $\cH$ such that
$\underline{V}$ is a lift of $\underline{T}$, i.e., for each $j = 1, \dots, d$,
$$
V_j^* \cH \subset \cH \text{ and } V_j^*|_\cH = T_j^*.
$$
It is known that a $K$-contraction $\underline{T}$ has a $K$-unitary dilation if and only if
$\underline{T}$ has a $K$-isometric lift.  In practice $K$-isometric lifts are easier to work with,
so in the sequel we shall only deal with $K$-isometric lifts.
This point has been made in a number of places (see e.g.\ the introduction in \cite{BS-failure}).

We define a couple of terminologies here. To add flexibility to the construction of such lifts, we often drop the requirement that $\cH$ be a subspace of
$\cK_+$ but instead require only an isometric identification map $\Pi \colon \cH \to \cK_+$.  We summarize the precise language which we shall
be using in the following definition.

\begin{definition}  \label{D:terminology}  We say that $(\Pi, \cK_+,  \underline{S} = (S_1, \dots, S_d))$
is a {\em lift} of $\underline{T} = (T_1, \dots, T_d)$ on $\cH$ if
\begin{itemize}
\item $\Pi \colon \cH \to \cK_+$ is isometric, and
\item $S_j^* \Pi h = \Pi T_j^* h$ for all $h \in \cH$ and $j=0,1,2,\dots, d$.
\end{itemize}

A lift $(\Pi, \cK_+, \underline{S})$  of $\underline{T}$ is said to be {\em minimal} if
$$
 \cK_+ = \overline{\operatorname{span}} \{S_1^{m_1} S_2^{m_2} \cdots S_d^{m_d} h \colon h \in \cH, m_j \ge 0\}.
$$

Two lifts $(\Pi, \cK_+, \underline{S})$ and $(\Pi,' \cK'_+, \underline{S}' )$  of the same $(T_1, \dots, T_d)$ are said to be {\em unitarily equivalent} if there
is a unitary operator $\tau \colon \cK_+ \to \cK'_+$ so that
$$
\tau S_j = S_j' \tau \text{ for each } j =1, \dots, d, \quad \text{and}\quad\tau \Pi = \Pi'.
$$
\end{definition}

It is known (see Chapter I of \cite{Nagy-Foias}) that when $K=\overline{\mathbb{D}}$, any two minimal isometric lifts of a given contraction are unitarily equivalent. However, minimality in several variables does not imply uniqueness, in general. For example, two minimal $\overline{\mathbb{D}}^2$-isometric lifts need not be unique (see \cite{Timotin-Li}).

Instead of $\overline{\mathbb{E}}$-contraction, the terminology {\em tetrablock contraction} was used in \cite{Tirtha14}.
We follow this terminology.

\subsection{Fundamental operators}
For our study of commutative contractive tuples $\underline{T} = (T_1, \dots, T_d)$,
we shall have use for the following theorem concerning $\Gamma$-contractions. We refer back to Section \ref{S:notation} for other notational conventions.

\begin{thm}\label{Thm:B-P-SR}
Let $(S,T)$ be a $\Gamma$-contraction on a Hilbert space $\cH$. Then
\begin{enumerate}
  \item[(1)] {\rm (See \cite[Theorem 4.2]{B-P-SR}.) }There exists a unique operator $F\in\cB(\cD_T)$ with $\nu(F)\leq 1$ such that
  $$
  S-S^*T=D_TFD_T.
  $$
  \item[(2)] {\rm (See \cite[Lemma 4.1]{Tirtha14}.)} The operator $F$ in part (1) above is  the unique solution $X = F$
  of the operator equation
  $$
  D_TS=X D_T+  X^*D_T T.
  $$
\end{enumerate}
\end{thm}

This theorem has been a major influence on further developments in the theory of both $\Gamma$-contractions \cite{BPJOT, BS-PRIMS} and
tetrablock contractions \cite{Tirtha14, BS-CAOT, SauNYJM}. The unique operator $F$ in Theorem \ref{Thm:B-P-SR} is called the
{\em fundamental operator} of the $\Gamma$-contraction $(S,T)$.

\subsection{Models for commutative isometric tuples}   \label{S:BCL}
The following result  of Berger, Coburn and Lebow  for commutative-tuples of isometries is a fundamental stepping stone
for our study of commutative-tuples of contractions.

\begin{thm}  \label{Thm:BCL}
  Let $d\geq2$ and $(V_1,V_2,\dots,V_d)$ be a $d$-tuple of commutative  isometries acting on $\cK$. Then there exit Hilbert spaces
  $\cF$ and $\cK_u$, unitary operators $U_1, \dots, U_d$ and projection operators $P_1, \dots, P_d$ on $\cF$,
  commutative  unitary operators $W_1, \dots, W_d$ on $
  \cK_u$, such that $\cK$ can be decomposed as
   \begin{align}\label{VWold}
   \cK=H^2(\cF)\oplus\cK_u
   \end{align}
and with respect to this decomposition
\begin{align}  \label{BCL1}
&V_j = M_{U_jP_j^\perp+zU_jP_j}\oplus W_j,\;
V_{(j)}=M_{P_jU_j^*+zP_j^\perp U_j^*}\oplus W_{(j)}  \text{ for }1 \le j \le d,\\
&\text{ and }  V =V_1V_2\cdots V_d= M_z\oplus W_1W_2\cdots W_d.\notag
\end{align}
\end{thm}

\begin{proof}  See Theorem 3.1 in \cite{BCL} as well as \cite[Section 2]{BDF1} for a different perspective.
\end{proof}

\begin{definition}  \label{D:BCLtuple}
Given two Hilbert spaces $\cF$, $\cE$, $d$ projections $P_1,P_2,\dots,P_d$ in $\cB(\cF)$, $d$ unitaries $U_1,U_2,\dots,U_d$ in $\cB(\cF)$, and
$d$ commuting unitaries $W_1,W_2,\dots,W_d$ in $\cB(\cE)$, the tuple
$$
(\cF,\cE,P_j,U_j,W_j)_{j=1}^d
$$
will be referred to as a {\em BCL tuple}. We shall call the tuple of isometries acting on $H^2(\cF)\oplus\cE$ given by
\begin{align}\label{noncomBCLmodel}
(M_{U_1P_1^\perp+zU_1P_1}\oplus W_1,M_{U_2P_2^\perp+zU_2P_2}\oplus W_2,\dots,M_{U_dP_d^\perp+zU_dP_d}\oplus W_d)
\end{align}
the {\em BCL model} associated with the BCL tuple $(\cF,\cE,P_j,U_j,W_j)_{j=1}^d$.
\end{definition}

\begin{remark}
If $P$ and $U$ are a projection and a unitary acting on a Hilbert space $\cF$, then one can check that the multiplication operator
$M_{U(P^\perp+zP)}$ acting on $H^2(\cF)$ is an isometry. It should however be noted that given $d$ projections $P_1,P_2,\dots,P_d$ and
unitaries $U_1,U_2,\dots,U_d$ on $\cF$, the tuple of isometries
$$
(M_{U_1P_1^\perp+zU_1P_1},M_{U_2P_2^\perp+zU_2P_2},\dots,M_{U_dP_d^\perp+zU_dP_d})
$$
need not be commutative, in general. Necessary conditions for such a tuple of isometries to be commuting are
given in Theorem 3.2 of \cite{BCL}:
\begin{align}
& U_1U_2\cdots U_d= I_{\cF}, \quad U_iU_j=U_jU_i \text{ for } 1\leq i,j \leq d,  \notag \\
& P_{j_1}+U_{j_1}^*P_{j_2}U_{j_1}+U_{j_2}^*U_{j_1}^*P_{j_3}U_{j_1}U_{j_2}+\cdots+U_{j_{d-1}}^*\cdots U_{j_1}^*P_{j_n}U_{j_1}\cdots U_{j_{d-1}}=I_{\cF} \notag \\
& \quad \text{ for }(j_1,j_2,\dots,j_d)   \in S_d     \label{NecConds}
\end{align}
where  $S_d$ is the permutation group on $d$ indices $\{1,2,\dots,d\}$.
 When $d=2$, these  necessary conditions  \eqref{NecConds}  simplify to
 \begin{align}\label{2NecCondsBCL}
 U_2=U_1^* \text{ and }P_2=I_{\cF}-P_1.
\end{align}
which turn out to be  sufficient as well for the $d=2$ case.
\end{remark}

It is well known that an arbitrary family of commutative isometries has a commutative unitary extension
(see \cite[Proposition I.6.2]{Nagy-Foias}).
The Berger-Coburn-Lebow model for commutative isometries gives  some additional information
regarding such extensions.

\begin{lemma}\label{special-ext}
Let $\underline{V}=(V_1,V_2,\dots, V_d)$ be a $d$-tuple of commutative  isometries on a Hilbert space $\mathcal{H}$ and $V=V_1V_2\cdots V_d$. Then $\underline{V}$ has a commutative  unitary extension $\underline{Y}=(Y_1,Y_2,\dots, Y_d)$ such that $Y=Y_1Y_2\cdots Y_d$ is the minimal unitary extension of $V$.
\end{lemma}

\begin{proof}
See \cite[Theorem 3.6]{BCL}.
\end{proof}

\subsection{Canonical commutative unitary tuple associated with a commutative tuple of contractions}  \label{S:Douglas}

Let $(T_1,T_2,\cdots,T_d)$ be a $d$-tuple of commutative contractions on a Hilbert space $\cH$ and $T=T_1T_2\cdots T_{d}$. Since $T$ is a contraction, the sequence $T^nT^{* n}$ converges in the strong operator topology. Let $Q$ be the positive semidefinite square root of
the limit operator, so
\begin{align}\label{Q}
Q^2:=\operatorname{SOT-}\lim T^nT^{* n}.
\end{align}
Then the operator $X^*:\overline{\operatorname{Ran}}\;Q\to\overline{\operatorname{Ran}}\;Q$ defined densely by
\begin{align}\label{theX}
X^*Q=QT^*,
\end{align}
is an isometry because for all $h\in\cH$,
\begin{align}\label{Xisoprf}
\langle Q^2h,h\rangle=\lim_{n\to\infty}\langle T^{n}T^{* n}T^*h,T^*h\rangle=\langle QT^*h,QT^*h\rangle.
\end{align}
Let $W_D^*$ on $\mathcal{R}_D\supseteq \overline{\operatorname{Ran }}\;Q$ be the minimal unitary extension of $X^*$.
Define the  operator $\widehat \cO_{D_{T^*}, T^*}:\cH\to H^2(\cD_{T^*})$ as
\begin{align}\label{observ}
\widehat \cO_{D_{T^*}, T^*}(z) h=\sum_{n=0}^\infty z^nD_{T^*} T^{*n} h, \text{ for every }h\in\cH.
\end{align}
Then the operator $\Pi_{D}:\mathcal{H}\to H^2(\mathcal D_{T^*})\oplus \mathcal{R}_D$ defined by
\begin{align}\label{Pi-D}
\Pi_{D}( h)= \widehat \cO_{D_{T^*}, T^*}(z) h\oplus Q(h)=\sum_{n=0}^\infty z^nD_{T^*} T^{*n} h\oplus Qh
\end{align}
is an isometry and satisfies the  intertwining property
\begin{equation} \label{DougInt}
\Pi_{D}T^*=(M_z\oplus W_D)^*\Pi_{D}
\end{equation}
(see e.g.\ \cite[Section 4]{Doug-Dilation}.
We conclude that with the isometry $V_D$ defined on $\cK_D:=H^2(\mathcal D_{T^*})\oplus \mathcal{R}_D$ as
\begin{align}\label{Vd}
V_D:=M_z\oplus W_D,
\end{align} $(\Pi_D,\cK_D,V_D)$ is an isometric lift of $T$.  One can furthermore shows that this lift is minimal
as well (see \cite[Lemma 1]{Doug-Dilation}).

If we now recall that $T = T_1,T_2,\dots ,T_d$, we see that for all $h\in\mathcal{H}$ and $i=1,2,\dots, d$,
\begin{eqnarray*}
  \langle T_iQ^2T_i^*h,h\rangle = \lim\langle T^n(T_iT_i^*){T^*}^nh,h \rangle \leq
  \lim\langle T^n{T^*}^nh,h\rangle = \langle Q^2h,h\rangle.
\end{eqnarray*}
By the Douglas Lemma \cite{Douglas}, this implies that there exists a contraction $X_i^*$ such that
\begin{align}\label{TheXi}
X_i^*Q=QT_i^*.
\end{align}
Since $T = T_1,T_2,\dots ,T_d$, it is evident that $(X_1,X_2,\dots,X_d)$ is a commutative  tuple of contractions and that
$$
X_1^*\cdots X_{d}^*=X^*,$$
where $X^*$ is as in (\ref{theX}). Since $X^*$ is an isometry, so also is each $X_i$. By Lemma \ref{special-ext} we have
a commutative unitary extension $(W_{\partial1}^*,W_{\partial2}^*,\dots, W_{\partial d}^*)$ of $(X_1^*,X_2^*,\dots, X_{d}^*)$
on the same space $\mathcal{R}_D\supseteq \overline{\operatorname{Ran}}\,Q$, where the minimal unitary extension
$W_D^*$ of $X^*$ acts and
\begin{align}\label{prodUni}
W_D=W_{\partial1}W_{\partial2}\cdots W_{\partial d}.
\end{align}
Note that this means
\begin{eqnarray}\label{TheSpace-R}
\mathcal{R}_D=\overline{\text{span}}\{W_D^{ n}x: x\in \overline{\operatorname{Ran}}\,Q\text{ and }n\geq0\}.
\end{eqnarray}
The tuple
\begin{equation}  \label{canonicalWpartial}
\underline{W}_\partial:=(W_{\partial1},W_{\partial2},\dots, W_{\partial d})
\end{equation}
 will be referred to as the {\em canonical commutative unitary tuple} associated with $(T_1,T_2,\dots,T_d)$. We next show that the canonical tuple of commutative unitary operators is uniquely determined by the tuple $(T_1,T_2,\dots,T_d)$.

\begin{lemma}\label{Lem:!Wparsi}
Let $\underline{T}=(T_1,T_2,\dots,T_d)$ on $\mathcal{H}$ and $\underline{T'}=(T_1',T_2',\dots,T_d')$ on $\mathcal{H'}$ be two tuples of commutative  contractions. Let $\underline{W}_\partial=(W_{\partial1},W_{\partial2},\dots,W_{\partial d})$ on $\mathcal{R}_D$ and $\underline{W}_{\partial}'=(W_{\partial1}',W_{\partial2}',\dots,W_{\partial d}')$ on $\mathcal{R}_D'$ be the respective commutative  tuples of unitaries obtained from $\underline{T}$ and $\underline{T}'$ as above, respectively. If $\underline{T}$ is unitarily equivalent to $\underline{T}'$ via the unitary similarity $\phi \colon \mathcal{H} \to \mathcal{H'}$, then so are $\underline{W}_{\partial}$ and $\underline{W}_{\partial}'$ via the induced unitary transformation $\tau_\phi \colon \mathcal{R}_D\to \mathcal{R}_D'$ determined by $\tau_\phi \colon W_D^nQh \to W_D'^n Q'\phi h$.  In particular, if $\underline{T} = \underline{T}'$, then $\underline{W}_{\partial} = \underline{W}_{\partial}'$.
\end{lemma}

\begin{proof}
That the tuples $\underline{W}_\partial$ and $\underline{W}_\partial'$ are obtained from $\underline{T}$ and $\underline{T}'$ respectively
means that
\begin{align}  & W_{\partial j}^*Q=QT_j^*, \;W_{\partial j}'^*Q'=Q'T_j'^* \text{ for each } j=1,2,\dots,d, \notag \\
& W_D=\prod_{j=1}^{\infty}W_{\partial j},\quad W_D'=\prod_{j=1}^{\infty}W_{\partial j}',  \label{ObtnMeans}
\end{align}
where $Q^2=\operatorname{SOT-lim}_{n \to \infty} T^n T^{*n}$ and $Q'^2=\operatorname{SOT-lim}_{n \to \infty} T'^n T'^{*n}$ with $T=T_1T_2\cdots T_d$ and $T'=T_1'T_2'\cdots T_d'$.  We shall show that set of equations (\ref{ObtnMeans}) is all that is needed to prove the lemma.

So suppose that the tuples $\underline{T}$ and $\underline{T}'$ are unitarily equivalent via the unitary similarity $\phi \colon \mathcal{H} \to \mathcal{H'}$. By definitions of $Q$ and $Q'$, it is easy to see that $\phi$ intertwines $Q$ and $Q'$ also and hence $\phi$ takes $\overline{\operatorname{Ran}}\,Q$ onto $\overline{\operatorname{Ran}}\,Q'$. By (\ref{ObtnMeans}) it follows that $\phi$ intertwines $W_{\partial j}^*|_{\overline{\operatorname{Ran}}\,Q}$ and $W_{\partial j}'^*|_{\overline{\operatorname{Ran}}\,Q'}$ for each $j=1,2,\dots, d$. Now remembering the formula (\ref{TheSpace-R}) for the spaces $\mathcal{R}_D$ and $\mathcal{R}_D'$, we define $\tau_\phi:\mathcal{R}_D\to\mathcal{R}_D'$ by
$$
\tau_\phi: W_D^{ n}x\mapsto W_D'^{ n}\phi x, \text{ for every } x\in \overline{\operatorname{Ran}}\,Q \text{ and } n\geq 0
$$
and extend linearly and continuously. It is evident that $\tau_\phi$ is unitary and intertwines $W_D$ and $W_D'$.
For a non-negative integer $n$, $j=1,2,\dots, d$ and $x$ in $\overline{\operatorname{Ran}}\,Q$, we compute
\begin{align*}
\tau_\phi W_{\partial j}(W_D^{n}x)  & =\tau_\phi W_D^{n+1}{\prod_{j\neq i=1}^{d} W_{\partial i}^*}x  =
W_D'^{ n+1}\phi\left({\prod_{j\neq i=1}^{d} W_{\partial i}^*}x\right) \\
& =W_D'^{n+1}{\prod_{j\neq i=1}^{d} W_{\partial i}'^*}\phi x=W_{\partial1}'W_D'^{n}\phi x=W_{\partial j}'\tau_\phi(W_D^{ n}x).
\end{align*}
and the lemma follows.
\end{proof}

\section{Fundamental operators for a tuple of commutative  contractions}   \label{S:fundamental}
The following result reduces the study of commutative contractive $d$-tuples to the study of a family of $\Gamma$-contractions.
This enables us to apply the substantial body of existing results concerning $\Gamma$-contractions to the study of
commutative contractive operator-tuples.

\begin{proposition}\label{Prop:Gamma&E-op}
Let $d\geq 2$ and $\underline{T}=(T_1,T_2,\dots,,T_d)$ be a $d$-tuple of commutative
contractions on a Hilbert space $\cH$ and let $T=T_1 \cdots T_d$.  Then for each $j=1,2,\dots, d$ and
$w\in \overline{\bD}$, the pair
\begin{align}\label{Gamma&E-op}
(S_{j}(w),T(w)):=(T_j+wT_{(j)},wT)
\end{align}is a $\Gamma$-contraction, where $T_{(j)}:=T_1\dots T_{j-1}T_{j+1}\dots T_{d}$.
\end{proposition}
\begin{proof}
Note that for each $j=1,2,\dots, d$ and $w\in \overline{\bD}$, the pair $(S_{j}(w),T(w))$ is actually the symmetrization of two commutative  contractions, viz., $T_j$ and $wT_{(j)}$. Since every such pair is a $\Gamma$-contraction, the result follows.
\end{proof}

Proposition \ref{Prop:Gamma&E-op} allows us to apply the $\Gamma$-contraction theory to obtain fundamental
operators associated with a $d$-tuple of commutative  contractions. This is the main result of this section.

\begin{thm}\label{fund-existence}
Let $d\geq 2$ and $\underline{T}=(T_1,T_2,\dots,,T_d)$ be a $d$-tuple of commutative  contractions on a Hilbert space $\cH$  and let $T=T_1T_2\cdots T_d$. Then
\begin{enumerate}

  \item For each $i=1,2,\dots,d$, there exist unique bounded operators $F_{i1},F_{i2}\in\cB(\cD_T)$ with
  $\nu(F_{i1}+wF_{i2})\leq1$  for all $w\in\overline{\bD}$ such that
  \begin{align}
   &  T_i-T_{(i)}^*T=D_TF_{i1}D_T,  \notag \\
   &  T_{(i)}-T_i^*T=D_TF_{i2}D_T.
  \label{En-fundeqn}
  \end{align}

  \item For each $i=1,2,\dots,d$, the pair $(F_{i1},F_{i2})$ as in part (1) is the unique solution
  $(X_{i1}, X_{i2}) = (F_{i1}, F_{i2})$ of the system of operator equations
  \begin{align}
  &  D_T  T_i = X_{i1}D_T+X_{i2}^*D_T T,  \notag \\
  &  D_T  T_{(i)} = X_{i2}D_T+X_{i1}^* D_T  T.
  \label{definingFs}
  \end{align}
\end{enumerate}
\end{thm}

\begin{proof}
 Note first that Proposition \ref{Prop:Gamma&E-op} ensures us that for all $w\in\bT$, the pairs $(S_{i}(w),T(w)):=(T_i+wT_{(i)},wT)$ are $\Gamma$-contractions. Hence by Theorem \ref{Thm:B-P-SR} there exist operators $F_{i}(w)\in \cB(\cD_T)$ such that
$$
S_{i}(w)-S_{i}(w)^*T(w)=D_TF_{i}(w)D_T
$$
which in turn simplifies to
\begin{equation} \label{fund2}
(T_i-T_{(i)}^*T)+w(T_{(i)}-T_i^*T)=D_TF_{i}(w)D_T.
\end{equation}
Let us introduce the notation
$$
L_0 = T_i - T_{(i)}^* T, \quad L_1 = T_{(i)} - T_i^* T, \quad L(w) = L_0 + w L_1
$$
so that we can write \eqref{fund2} more compactly as
\begin{equation}  \label{compact}
L(w) = D_T F_i(w) D_T.
\end{equation}
Our goal is to show that then necessarily $F_i(w)$ has the pencil form
\begin{equation}  \label{F-pencil}
  F_i(w) = F_{i1} + w F_{i2}
\end{equation}
for some uniquely determined operators $F_{i1}$ and $F_{i2}$ in $\cB(\cD_T)$.
Note that we recover $L_0$ and $L_1$ from $L(w)$ via the formulas
$$
 L_0 = L(0), \quad L_1 = \frac{ L(w) - L(0)}{w} \text{ for any } w \in {\mathbb D} \setminus \{0\}.
$$
Since $L(0) = D_T F_i(0) D_T$, it is natural to set
\begin{equation}  \label{Fi0}
    F_{i1} = F_i(0).
\end{equation}
Similarly, since we recover $L_1$ from $L(w)$ via the formula
$$
L_1 = \frac{ L(w) - L(0) }{w} \text{ for any } w \in {\mathbb D} \setminus \{0\},
$$
it is natural to set
\begin{equation}   \label{Fi2}
  F_{i2} = \frac{ F_i(w) - F_i(0)}{w} \text{ for } w \in {\mathbb D} \setminus \{0\}.
\end{equation}
To see that the right-hand side of \eqref{Fi2} is independent of $w$, we note that
$$
 L_1 = \frac{ L(w) - L(0)}{w} = D_T \frac{ F_i(w) - F_i(0)}{w} D_T.
$$
Since $L_1$ is independent of $w$ and $(F_i(w) - F_i(0))/w \in \cB(\cD_T)$,  it follows that,
for any two points $w, w' \in {\mathbb D} \setminus \{0\}$, we have
$$
  D_T \bigg( \frac{ F_i(w) - F_i(0)}{w}  - \frac{ F_i(w') - F_i(0)}{w'} \bigg) D_T = 0.
$$
From the general fact
\begin{equation}   \label{fact}
X \in \cL(\cD_T), \, D_T X D_T = 0 \Rightarrow X = 0,
\end{equation}
it follows that $\frac{F_i(w) - F_i(0)}{w}  = \frac{ F_i(w') - F_i(0)}{w'} $ and hence $F_{i2}$ is well-defined by
\eqref{Fi2}. From the definitions we see that $L_0 + w L_1 = D_T (F_{i1} + w F_{i2}) D_T$ and hence,
again by the uniqueness statement \eqref{fact}, we have established that $F_i(w)$ has the pencil form
\eqref{F-pencil} as wanted.

Finally equations \eqref{En-fundeqn} now follow by equations coefficients in the pencil identity
$L(w) = D_T T(w) D_T$.

To prove part (2), we see by part (2) of Theorem \ref{Thm:B-P-SR} that for each $i=1,2,\dots,d$ and $w\in\mathbb{T}$, the operator $F_{i}(w)$ is the unique operator that satisfies
$$
D_TS_{i}(w)=F_{i}(w)D_T+F_{i}(w)^*D_TT(w).
$$
Hence it follows that for all $w\in\mathbb{T}$ we have
$$
D_T(T_i+wT_{(i)})=(F_{i1}+wF_{i2})D_T+w(F_{i1}+wF_{i2})^*D_TT.
$$
A comparison of the constant terms and the coefficients of $w$ gives the equations in (\ref{definingFs}).

The uniqueness part follows from that of
the function $F_i(w)$ as follows.  If $F_{i1}'$ and $F_{i2}'$ are operators on $\cD_T$ that satisfy (\ref{definingFs}), then setting $F_i(w)':=F_{i1}'+wF_{i2}'$ gives
$$
D_TS_{i}(w)=F_{i}(w)'D_T+F_{i}(w)'^*D_TT(w).
$$
By the uniqueness in part (2) of Theorem \ref{Thm:B-P-SR}, we conclude $F_{i}(w)=F_i(w)'$ proving $F_{i1}=F_{i1}'$ and $F_{i2}=F_{i2}'$
for all $i=1,2,\dots,d$.
\end{proof}

\begin{definition}  \label{D:fund-op}
For a $d$-tuple $\underline{T}=(T_1,T_2,\dots,T_d)$ of commutative  contractions on a Hilbert space $\cH$, the unique operators $\{F_{i1},F_{i2}:i=1,2,\dots,d\}$ obtained in Theorem \ref{fund-existence} are called the {\em fundamental operators} of $\underline{T}$.
The fundamental operators of the adjoint tuple $\underline{T}^*=(T_1^*,T_2^*,\dots,T_d^*)$ will be denoted by $\{G_{i1},G_{i2}:i=1,2,\dots,d\}$.
\end{definition}

The following is a straightforward consequence of Theorem \ref{fund-existence}.

\begin{corollary}
Let $\underline{T}=(T_1,T_2,\dots,T_d)$ be a $d$-tuple of commutative  contractions on a Hilbert space $\cH$. Then the fundamental operators $\{G_{i1},G_{i2}:i=1,\dots,d\}$ of the adjoint tuple $\underline{T}^*=(T_1^*,T_2^*,\dots,T_d^*)$ are the unique operators  satisfying the systems
of  equations
\begin{align}\label{En-fundeqn*}
  \begin{cases}
    T_i^*-T_{(i)}T^*=D_{T^*}G_{i1}D_{T^*} \text{ and}\\
    T_{(i)}^*-T_iT^*=D_{T^*}G_{i2}D_{T^*}
    \end{cases}
  \end{align}
and
  \begin{align}\label{definingFs*}
  \begin{cases}
    D_{T^*}T_i^*=G_{i1}D_{T^*}+G_{i2}^*D_{T^*}T^* \text{ and}\\
    D_{T^*}T_{(i)}^*=G_{i2}D_{T^*}+G_{i1}^*D_{T^*}T^*
  \end{cases}
  \end{align}
for each $i=1,2,\dots,d$:
\end{corollary}

We next note some additional properties of the fundamental operators. These properties will not be used in this paper but  are of interest
in their own right.

\begin{proposition}\label{further-prop}
Let $\underline{T}=(T_1,T_2,\dots,T_d)$ be a $d$-tuple of commutative contractions on a Hilbert space $\cH$ and $T=T_1T_2\cdots T_d$.
Let $\{F_{j1},F_{j2}:j=1,\dots,d\}$ and $\{G_{j1},G_{j2}:j=1,\dots,d\}$ be the fundamental operators of $\underline{T}$ and $\underline{T}^*$,
respectively. Then for each $j=1,2,\dots,d$,
\begin{enumerate}
  \item $PF_{j1}=G_{j1}^*T|_{\mathcal{D}_T}$;
  \item $D_TF_{j1}=(T_jD_T-D_{T^*}G_{j2}T)|{\mathcal{D}_T}, \;D_TF_{j2}=(T_{(j)}D_{T}-D_{T^*}G_{j1}T)|{\mathcal{D}_T}$;
  \item $(F_{j1}^*D_TD_{T^*}-F_{j2}T^*)|_{\mathcal{D}_{T^*}}=D_TD_{T^*}G_{j1}-T^*G_{j2}^* \text{ and }\\(F_{j2}^*D_TD_{T^*}-F_{j1}T^*)|_{\mathcal{D}_{T^*}}=D_TD_{T^*}G_{j2}-T^*G_{j1}^*$.
\end{enumerate}
\end{proposition}

\begin{proof}
Let $(T_1,T_2)$ be a commutative pair of contractions on a Hilbert space $\cH$. We claim that the triple $(T_1,T_2,T_1T_2)$ is a tetrablock contraction or equivalently the closure of the tetrablock domain $\mathbb E$ as in (\ref{E-def1}) is a spectral set for $(T_1,T_2,T_1T_2)$. Let $\pi_{\mathbb{D}^2,\mathbb{E}}:\mathbb{D}^2\to\mathbb{E}$ be the map defined by
\begin{align}\label{BidisktoTetra}
\pi_{\mathbb{D}^2,\mathbb{E}} \colon(z_1,z_2)\mapsto (z_1, z_2, z_1z_2)
\end{align}
and let $f$ be any polynomial in three variables. Then by And\^o's theorem
\begin{align*}
\|f(T_1,T_2,T_1T_2)\|=\|f\circ\pi_{\mathbb{D}^2,\mathbb{E}}(T_1,T_2)\|\leq \|f\circ\pi_{\mathbb{D}^2,\mathbb{E}}\|_{\infty,\mathbb{D}^2}\leq \|f\|_{\infty,\mathbb{E}}.
\end{align*}
Therefore the triple $(T_1,T_2,T_1T_2)$ is a tetrablock contraction whenever $(T_1,T_2)$ is a commutative pair of contractions.

Thus, given a $d$-tuple $\underline{T}=(T_1,T_2,\dots,T_d)$ of commutative contractions on a Hilbert space $\cH$, there is an associated
 family of tetrablock contractions, viz.,
\begin{align} \label{FamilyTetra}
(T_j,T_{(j)},T),\quad j=1,2,\dots,d.
\end{align}
Hence parts (1), (2) and (3) of Proposition \ref{further-prop} follow immediately from Lemmas 8, 9 and 10 of \cite{BS-CAOT}, respectively.
\end{proof}

\begin{remark}
We note that a result parallel to Theorem \ref{fund-existence} appears in the theory of tetrablock contractions
(see \cite[Theorem 3.4]{Tirtha14}, namely:  {\em if $(A,B,T)$ is a tetrablock contraction, then  \cite{Tirtha14} that
there exists two bounded operators $F_1,F_2$ acting on $\cD_T$ with $\nu(F_1+zF_2)\leq 1$, for every $z\in\overline{\mathbb{D}}$ such that}
$$
  A-B^*T=D_TF_1D_T \text{ and } B-A^*T=D_TF_2D_T.
$$
Moreover,  Corollary 4.2 in \cite{Tirtha14} shows that {\em $F_1,F_2$ are the unique operators $(X_1, X_2)$ such that}
$$
D_T A=X_1D_T+X_2^*D_T T,   \quad D_T B=X_2 D_T+X_1^* D_T T.
$$
These unique operators $F_1,F_2$ are called the {\em fundamental operators} of the tetrablock contraction $(A,B,T)$.

Furthermore, it is possible to arrive at the result of Theorem \ref{fund-existence} via applying these
results to the special tetrablock contractions (\ref{FamilyTetra}).  Our proof of Theorem \ref{fund-existence} instead
relies only on the properties of fundamental operators for $\Gamma$-contractions.
\end{remark}

\section{Joint Halmos dilation of fundamental operators}  \label{S:Halmos}

The following notion of dilation for the case $d=1$ goes back to a 1950 paper of  Halmos \cite{Halmos}, hence our term
{\em joint Halmos dilation}.

\begin{definition}
For a tuple $\underline{A}=(A_1,A_2,\dots, A_d)$ of operators on a Hilbert space $\mathcal{H}$, a tuple $\underline{B}=(B_1,B_2,\dots,B_d)$
of operators acting on a Hilbert space $\mathcal{K}$ containing $\mathcal{H}$ is called a {\em{joint Halmos dilation}} of $\underline{A}$, if
there exists an isometry $\Lambda:\mathcal{H}\to\mathcal{K}$ such that $A_i=\Lambda^*B_i\Lambda$ for each $i=1,2,\dots,d$.
\end{definition}

\begin{lemma}\label{isometries}
Let $d\geq 2$ and $(T_1,T_2,\dots,T_d)$ be a $d$-tuple of commutative  contractions on $\cH$ and $T=T_1T_2\cdots T_d$.
\begin{enumerate}
\item Let $\alpha=(j_1,\dots,j_k)$ be a $k$-tuple such that $1\leq j_1<j_2<\cdots<j_k\leq d$. Consider the $k$-tuple $(T_{j_1},T_{j_2},\dots,T_{j_k})$ and define $T_{ \alpha} = T_{j_1}\cdots T_{j_k}$.
Let $\Delta_\alpha:\cD_{T_{\alpha}}\to\cD_{T_{j_1}}\oplus\cD_{T_{j_2}}\oplus\cdots \oplus \cD_{T_{j_k}}$ be the operator defined by
\begin{align}
& \Delta_\alpha \colon D_{T_{\alpha}}h  \mapsto  \notag \\
& \quad D_{T_{j_1}}T_{j_2}\cdots T_{j_k}h \oplus D_{T_{j_2}}T_{j_3} \cdots T_{j_k}h\oplus\cdots\oplus D_{T_{j_{k-1}}}T_{j_k}h
\oplus D_{T_{j_k}}h
\label{Deltalpha}
\end{align}
for all $h\in\cH$. Then $\Delta_\alpha$ is an isometry.

\item For each $j=1,2,\dots, d$, the operator
$$
\Lambda_j \colon \cD_T\to\cD_{T_1}\oplus\cD_{T_2}\oplus\cdots\oplus\cD_{T_d}
$$
given by
\begin{align}\label{Lambdaj}
\Lambda_j \colon D_Th \mapsto D_{T_j}T_{(j)}h\oplus \Delta_{(j)} D_{T_{(j)}}h,\text{ for all }h\in \cH
\end{align}
is an isometry, where $\Delta_{(j)}$ for the tuple ${(j)}=(1,\dots,j-1,j+1,\dots,d)$ is as in \eqref{Deltalpha}.

\item For each $j=1,2\dots, d$, the operator $U_j^*\colon \operatorname{Ran \;\Lambda_j}\to\cD_{T_1}\oplus\cD_{T_2}\oplus\cdots\oplus\cD_{T_d}$ defined by
\begin{align}\label{Uj*}
U_j^*\colon D_{T_j}T_{(j)}h\oplus \Delta_{(j)} D_{T_{(j)}}h\mapsto D_{T_j}h\oplus \Delta_{(j)} D_{T_{(j)}}T_jh \text{ for all }h\in \cH
\end{align}
is an isometry.
\item After possibly enlarging the Hilbert space $\cD_{T_1}\oplus\cD_{T_2}\oplus\cdots\oplus\cD_{T_d}$ to a larger Hilbert space
$$
 \cF:= \cD_{T_1}\oplus\cD_{T_2}\oplus\cdots\oplus\cD_{T_d} \oplus \cE
$$
for some auxiliary Hilbert space $\cE$,
\begin{itemize}
\item[(a)]the isometries $U_j^*$ in part (3) can be extended to be unitary operators on $\cF$ (still denoted as $U_j^*$).

\item[(b)] for each $j=1,2\dots, d$, there exists a unitary operator $\tau_j$ on $\cF$ such that
\begin{align}\label{varphi-j}
  \tau_j\Lambda_j=\Lambda_1,
\end{align}
where  $\tau_1=I_{\cF}$.
\end{itemize}
\end{enumerate}
\end{lemma}

\begin{proof}[Proof of Part (1)]
Note that the norm of the vector on the RHS of \eqref{Deltalpha} is
\begin{equation}  \label{RHSvectornorm}
\|D_{T_{j_1}}T_{j_2}\cdots T_{j_k}h\|^2 + \|D_{T_{j_2}}T_{j_3} \cdots T_{j_k}h\|^2+\cdots+\|D_{T_{j_{k-1}}}T_{j_k}h\|^2+\|D_{T_{j_k}}h\|^2
\end{equation}
Making use of the general fact that
if $T$ is a contraction on a Hilbert space $\cH$, then $\|D_Th\|^2=\|h\|^2-\|Th\|^2$ for every $h\in\cH$, we can convert \eqref{RHSvectornorm}
to the telescoping sum
\begin{align*}
&   (\|T_{j_2}\cdots T_{j_k}h\|^2-\|T_{j_1}T_{j_2}\cdots T_{j_k}h\|^2)+(\|T_{j_3} \cdots T_{j_k}h\|^2-\|T_{j_2}T_{j_3} \cdots T_{j_k}h\|^2)  \\
& \quad \quad \quad \quad  +\cdots+(\|T_{j_k}h\|^2-\|T_{j_{k-1}}T_{j_k}h\|^2)+(\|h\|^2-\|T_{j_k}h\|^2). \\
 &  = \|h\|^2-\|T_{j_1}T_{j_2}\cdots T_{j_k}h\|^2=\|h\|^2-\|T_{\alpha}h\|^2=\|D_{T_{\alpha}}h\|^2.
\end{align*}This shows that $\Delta_\alpha$ is an isometry.

{\em Proof of Part (2).} Use the fact that $\Delta_{(j)}$ is an isometry by Part (1) to get
\begin{align*}
& \|D_{T_j}T_{(j)}h\|^2+\|\Delta_{(j)} D_{T_{(j)}}h\|^2=\|D_{T_j}T_{(j)}h\|^2+\|D_{T_{(j)}}h\|^2\\
& \quad =(\|T_{(j)}h\|^2-\|Th\|^2)+(\|h\|^2-\|T_{(j)}h\|^2)=\|D_Th\|^2.
\end{align*}

{\em Proof of Part (3).} By a similar computation as done in Part (2), one can show that the norms of the vectors $D_{T_j}T_{(j)}h\oplus \Delta_{(j)} D_{T_{(j)}}h$ and $D_{T_j}h\oplus \Delta_{(j)} D_{T_{(j)}}T_jh$ are the same for every $h$ in $\cH$.

{\em Proof of Part (4).} Denote by $\cD$ the Hilbert space $\cD_{T_1}\oplus\cD_{T_2}\oplus\cdots\oplus\cD_{T_d}$. If for each $j=1,2,\dots,d$,
\begin{align}\label{1stpt}
\nonumber&\dim(\cD\ominus\overline{\{D_{T_j}T_{(j)}h\oplus \Delta_{(j)} D_{T_{(j)}}h:h\in\cH\}})\\
&=\dim(\cD\ominus\overline{\{D_{T_j}h\oplus \Delta_{(j)} D_{T_{(j)}}T_jh:h\in\cH\}}),
\end{align}
then clearly the isometric operators $U_j^*$ defined as in (\ref{Uj*}) extend to  unitary operators on $\cD_{T_1}\oplus\cD_{T_2}\oplus\cdots\oplus\cD_{T_d}$.   Then we may define unitary operators $\tau_j$ on $\cD$ by
\begin{equation}  \label{tauj-def}
  \tau_j = U_1U_j^*.
\end{equation}
(so in particular $\tau_1 = I_\cD$).
Note next that
\begin{align*}
\tau \colon \Lambda_j  D_Th =  D_{T_j}T_{(j)}h\oplus \Delta_{(j)} D_{T_{(j)}}h\mapsto D_{T_1}T_{(1)}h\oplus \Delta_{(1)}
D_{T_{(1)}}h=\Lambda_1D_Th
\end{align*}
Hence  $\tau_j$ is a well-defined unitary operator on all of $\cD$ satisfying the intertwining relation \eqref{varphi-j}.
 If any of the equalities in (\ref{1stpt}) does not
hold, then we add an infinite dimensional Hilbert space $\cE$ to $\cD_{T_1}\oplus\cD_{T_2}\oplus\cdots\oplus\cD_{T_d}$ so that (\ref{1stpt})
does hold with $\cD$ replaced by
$$
\cF:=\cD_{T_1}\oplus\cD_{T_2}\oplus\cdots\oplus\cD_{T_d}\oplus\cE.
$$
This proves (4).
\end{proof}

\noindent
{\bf{Notation:}} For the adjoint tuple $(T_1^*,T_2^*,\dots, T_d^*)$, the symbols $\cF,\Lambda_j,U_j,\tau_j$ introduced in Lemma \ref{isometries} will
be changed to $\cF_{j*},\Lambda_{j*},U_{j*}$ and $\tau_{j*}$, respectively. In addition to this, we denote by $P_j$ and $P_{j*}$ the projections of
$\cF$ and $\cF_*$ onto $\cD_{T_j}$ and $\cD_{T_{j*}}$, respectively.

\begin{definition}
Let $d\geq 2$ and $\underline{T}=(T_1,T_2,\dots, T_d)$ be a $d$-tuple of commutative  contractions. Let $\cF,\Lambda_j,U_j$ be as in
Lemma \ref{isometries}, and for $j=1,2,\dots,d$ let $P_j$ denote the projection of $\cF$ onto $\cD_{T_j}$. Then we say that the tuple
$(\cF, \Lambda_j,P_j,U_j)_{j=1}^d$ is an {\em And\^o tuple} for $\underline{T}$.
\end{definition}

\begin{thm}\label{Thm:HalmosDil}
Let $d\geq 2$, $\underline{T}=(T_1,T_2,\dots, T_d)$ be a $d$-tuple of commutative  contractions on a Hilbert space $\cH$, $(\cF, \Lambda_j,P_j,U_j)_{j=1}^d$ be an And\^o tuple for $\underline{T}$ and $\{F_{j1},F_{j2}:j=1,2,\dots,d\}$ be the fundamental operators for $\underline{T}$. Then
\begin{enumerate}
\item For each $j=1,2\dots,d$, the pair $(F_{j1},F_{j2})$ has a joint Halmos dilation to a commutative  pair of partial isometries, viz.,
\begin{align}\label{HalmosDil}
(F_{j1},F_{j2})=\Lambda_j^*(P_j^\perp U_j^*,U_jP_j)\Lambda_j.
\end{align}
\item With the unitaries $\tau_j$ as obtained in part (5) of Lemma \ref{isometries}, there is a joint Halmos dilation for the set $\{F_{j1},F_{j2}:j=1,2,\dots,d\}$, viz.,
\begin{align}\label{JHalmosDil}
(F_{j1},F_{j2})=\Lambda_1^*(\tau_j P_j^\perp U_j^*\tau_j^*,\tau_jU_jP_j\tau_j^*)\Lambda_1, \text{ for each }j=1,2,\dots, d.
\end{align}
\end{enumerate}
\end{thm}

\begin{proof}[Proof of (1):] The proof of this part uses the uniqueness of the fundamental operators. For $h,h'\in\cH$ we have
\begin{align*}
& \langle D_T\Lambda_j^*P_j^\perp U_j^*\Lambda_jD_Th,h'\rangle  \\
& \quad = \langle P_j^\perp U_j^*(D_{T_j}T_{(j)}h\oplus \Delta_{(j)}
D_{T_{(j)}}h),D_{T_j}T_{(j)}h'\oplus \Delta_{(j)} D_{T_{(j)}}h'\rangle\\
& \quad =\langle 0\oplus\Delta_{(j)} D_{T_{(j)}}T_jh,   D_{T_j}T_{(j)}h'\oplus \Delta_{(j)} D_{T_{(j)}}h'\rangle\\
& \quad =\langle D_{T_{(j)}}T_j h, D_{T_{(j)}}h'\rangle=\langle (T_j-T_{(j)}^*T)h,h'\rangle.
\end{align*}
Therefore $D_T\Lambda_j^*P_j^\perp U_j^*\Lambda_jD_T=T_j-T_{(j)}^*T$. By a similar computation one can
show that $D_T\Lambda_j^*U_jP_j\Lambda_jD_T=T_{(j)}-T_{j}^*T.$ By Theorem \ref{fund-existence}, the fundamental
operators are the unique operators satisfying these equations. Hence (\ref{HalmosDil}) follows.

{\em{Proof of (2):}} This follows from the property \eqref{varphi-j}  of the unitaries $\tau_j$:
$\tau_j\Lambda_j=\Lambda_1$, for each $j=1,2,\dots,d$. Using this in (\ref{HalmosDil}), we obtain (\ref{JHalmosDil}).
\end{proof}

\section{Non-commutative  isometric lift of tuples of commutative  contractions}  \label{S:nc-iso-lift}

As was mentioned in connection with the rational dilation problem in Section \ref{S:domains},  for $d \ge 3$
it can happen that
a ${\overline{\mathbb D}}^d$-contraction fails to have a ${\overline{\mathbb D}}^d$-isometric lift once $d \ge 3$,
unlike the case of $d=1$ and $d=2$.
Note that a ${\overline{\mathbb D}}^d$-contraction consists of a commutative $d$-tuple $\underline{T} = (T_1, \dots, T_d)$ of
contraction operators, while a ${\overline{\mathbb D}}^d$-isometry consists of a commutative $d$-tuple of isometries.
Here we show that,
even for the case of a general $d \ge 3$, a ${\overline{\mathbb D}}^d$-contraction always has a
(in general noncommutative) isometric
lift $\underline{V} = (V_1, \dots, V_d)$.

Let $T$ be a contraction on a Hilbert space $\cH$. Sch\"affer \cite{Schaffer} showed that the following $2\times2$ block operator matrix
\begin{align}\label{SchDil}
V^S=
\begin{bmatrix}
  T & 0 \\
  {\bf e}_0^*D_T & M_z
\end{bmatrix}:\cH\oplus H^2(\cH)\to \cH\oplus H^2(\cH)
\end{align}is an isometry and hence a lift of the contraction $T$. Here, ${\bf e}_0 \colon H^2(\cH) \to \cH$ is the ``evaluation at zero" map:
${\bf e}_0 \colon g \mapsto g(0)$. For a given $d$-tuple $\underline{T} = (T_1, \dots, T_d)$ of contraction operators, let $V^S_j$ be the
Sch\"affer isometric lift of $T_j$, for each $1\leq j\leq d$. Then the tuple $\underline{V}^S:=(V_1^S,V_2^S,\dots,V_d^S)$ is an
(in general noncommutative) isometric lift of $\underline{T}$.

It is of interest to develop other constructions for such possibly noncommutative isometric lifts which have more
structure and provide additional insight.
Our next goal is to provide one such construction where the (possibly noncommutative) isometric lift of the given  commutative contractive $d$-tuple $\underline{T} = (T_1, \dots, T_d)$ has the form of a (possibly noncommutative) BCL model.  The
starting point for the construction is an And\^o tuple $(\cF_*, \Lambda_{j*},P_{j*},U_{j*})_{j=1}^d$ of $\underline{T}^*$,
with the BCL model for the lift then having the form
$$
(M_{\tau_{j*}(U_{j*}P_{j*}^\perp+zU_{j*}P_{j*})\tau_{j*}^*}\oplus W_{\partial j}^*)_{j=1}^d,
$$
where $\underline{W}_\partial:=(W_{\partial1},W_{\partial2},\dots, W_{\partial d})$ is the canonical commutative unitary tuple associated
with the commutative contractive $\underline{T}$ as in \eqref{canonicalWpartial}, and where $\tau_j$ ($j=1, \dots, d$) are unitaries acting on $\cF$ as in
Part (5) of Lemma \ref{isometries}. We first need a preliminary lemma.

\begin{lemma}\label{L:inflate-ID}
Let $d\geq 2$, $(T_1,T_2,\dots, T_d)$ be a $d$-tuple of commutative  contractions and $(\cF, \Lambda_j,P_j,U_j)_{j=1}^d$ be an
And\^o tuple for $(T_1,T_2,\dots, T_d)$. Then  the  operator identities
\begin{align}
& P_j^\perp U_j^* \Lambda_j D_T+P_j U_j^* \Lambda_j D_T T=   \Lambda_j D_T T_j,  \notag  \\
&  U_j P_j  \Lambda_j D_T+U_j P_j^\perp \Lambda_j D_T T =\Lambda_j D_T T_{(j)}   \label{inflate-ID}
\end{align}
hold for $j=1, \dots, d$.
\end{lemma}

\begin{proof}
  Let $j$ be some integer between $1$ and $d$, and $h$ be in $\cH.$ Then
  \begin{align*}
  & P_j^\perp U_j^*\Lambda_j D_Th+P_jU_j^*\Lambda_jD_T T h\\
  & = P_j^\perp U_j^*(D_{T_j}T_{(j)}h\oplus \Delta_{(j)} D_{T_{(j)}}h)+P_jU_j^*(D_{T_j}T_{(j)}Th\oplus \Delta_{(j)} D_{T_{(j)}}Th)\\
  & = P_j^\perp(D_{T_j}h\oplus \Delta_{(j)} D_{T_{(j)}}T_jh)+P_j(D_{T_j}Th\oplus \Delta_{(j)} D_{T_{(j)}}T_jTh)\\
  & = (0\oplus \Delta_{(j)} D_{T_{(j)}}T_jh) + (D_{T_j}Th\oplus 0)=D_{T_j}T_{(j)}T_jh\oplus \Delta_{(j)} D_{T_{(j)}}T_jh \\
  & =\Lambda_jD_T T_j h.
  \end{align*}
  The proof of the second equality in \eqref{inflate-ID}  is similar to that of the first one.
\end{proof}

\begin{thm}\label{Thm:noncomIsolift}
Let $\underline{T}=(T_1,T_2,\dots,T_d)$ be a $d$-tuple of commutative contractions, $\underline{W}_\partial:=(W_{\partial1}^*,W_{\partial2}^*,\dots, W_{\partial d}^*)$ be the canonical commutative unitary tuple associated with $\underline{T}$ as in \eqref{canonicalWpartial}, let $(\cF_*,\Lambda_{j*},
let P_{j*},U_{j*})_{j=1}^d$ be an And\^o tuple for $\underline{T}^*$, and let $T = T_1T_2\cdots T_d$. Then
\begin{enumerate}

\item For each $1\leq j \leq d$, define the isometries $\Pi_{j*}:\cH\to H^2(\cF_*)\oplus\cR_D$ as
\begin{align}
\Pi_{j*}h=(I_{H^2}\otimes \Lambda_{j*})\widehat{\cO}_{D_{T^*},T^*}h\oplus Qh.\label{Pi-j*}
\end{align}
Then the  identities
\begin{align}\label{AbstrctDil}
&  \Pi_{j*}T_{j}^*=(M_{U_{j*}P_{j*}^\perp+zU_{j*}P_{j*}}^*\oplus W_{\partial j}^*)\Pi_{j*} \notag \\
&  \Pi_{j*}T_{(j)}^*=(M^*_{P_{j*}U_{j*}^*+zP_{j*}^\perp U_{j*}^*}\oplus W_{(\partial j)}^*)\Pi_{j*}
\end{align}
hold for $1 \le j \le d$,
i.e. for each $j=1,2,\dots,d$,
\begin{align}\label{Dil-pair}
\big(\Pi_{j*},M_{U_{j*}P_{j*}^\perp+zU_{j*}P_{j*}}\oplus W_{\partial j},M_{P_{j*}U_{j*}^*+zP_{j*}^\perp U_{j*}^*}\oplus W_{(\partial j)},M_z\oplus W_D\big)
\end{align}
is an isometric lift of $(T_j,T_{(j)},T)$.

\item With the unitaries $\tau_{j*}$ obtained as in part (4) of Lemma \ref{isometries} applied to $(T_1^*,T_2^*,\dots,T_d^*)$, we have for each $j=1,2,\dots,d$,
\begin{align}\label{Dil}
\begin{cases}
  \Pi_{1*}T_j^*=(M^*_{\tau_{j*}(U_{j*}P_{j*}^\perp+zU_{j*}P_{j*})\tau_{j*}^*}\oplus W_{\partial j}^*)\Pi_{1*}\\
  \Pi_{1*}T_{(j)}^*=(M^*_{\tau_{j*}(P_{j*}U_{j*}^*+zP_{j*}^\perp U_{j*}^*)\tau_{j*}^*}\oplus W_{(\partial j)}^*)\Pi_{1*},
  \end{cases}
\end{align}
i.e. if we denote the projections $\tau_{j*}P_{j*}\tau_{j*}^*$ and unitaries $\tau_{j*}U_{j*}\tau_{j*}^*$  by $P_{j*}'$ and $U_{j*}'$, respectively, then the
$d$-tuple of (in general non-commutative) isometries
\begin{align}\label{Dil-tuple}
(M_{(U_{j*}'P_{j*}'^\perp+zU_{j*}'P_{j*}')}\oplus W_{\partial j})_{j=1}^d
\end{align}
is a lift of $(T_1,T_2,\dots,T_d)$ via the embedding $\Pi_{1*}:\cH\to H^2(\cF_*)\oplus\cR_D$.
\end{enumerate}
\end{thm}

\begin{proof}[Proof of part (1):] For every $h\in\cH$, we have for each $j=1,2,\dots, d$,
\begin{align*}
&(M_{U_{j*}P_{j*}^\perp+zU_{j*}P_{j*}}^*\oplus W_{\partial j}^*)\Pi_{j*}h\\
&=\sum_{n\geq 0}z^nP_{j*}^\perp U_{j*}^*\Lambda_{j*}D_{T^*}T^{* n}h+\sum_{n\geq 0}z^nP_{j*} U_{j*}^*\Lambda_{j*}D_{T^*}T^{* n+1}h\oplus W_{\partial j}^*Qh\\
&=\sum_{n\geq 0}z^n(P_{j*}^\perp U_{j*}^*\Lambda_{j*}D_{T^*}+P_{j*} U_{j*}^*\Lambda_{j*}D_{T^*}T^*)T^{* n}h\oplus QT_j^*h\\
&=\sum_{n\geq 0}z^n\Lambda_{j*}D_{T^*}T_j^*T^{* n}h\oplus QT_j^*h=\Pi_{j*}T_j^*h
\end{align*}
 where we use the first equation in \eqref{inflate-ID}) for the last step.

A similar computation using the second equation in (\ref{inflate-ID}) leads to the second equation in \eqref{AbstrctDil}.

{\em{Proof of part (2):}} It  follows from property  (\ref{varphi-j}) of $\tau_j$ and the definition (\ref{Pi-j*}) of $\Pi_{j*}$ that
$$((I_{H^2}\otimes \tau_{j*})\oplus I_{\cR})\Pi_{j*}=\Pi_{1*}, \text{ for each }j=1,2,\dots,d.$$ Using this in (\ref{AbstrctDil}) one obtains (\ref{Dil}).
\end{proof}

\begin{remark}
Note that for a given $d$-tuple $\underline{T}=(T_1,T_2,\dots,T_d)$ of commutative contractions, if there exists an
And\^o tuple of $\underline{T}^*$ such that the $d$-tuple of isometries given in (\ref{Dil-tuple}) is commutative, then
there exists a commutative isometric lift of $\underline{T}$. Therefore a priori, we have a sufficient condition for dilation in
$\overline{\mathbb{D}}^d$.
\end{remark}

\begin{remark} \label{R:Douglas-model}  Note that in the terminology of Definition \ref{D:terminology}, the context of part (2) of
Theorem \ref{Thm:noncomIsolift} the collection of objects $(\Pi_D, H^2(\cF_*) \oplus \cR_D, \underline{V})$, where we set
$\underline{V} = (V_1, \dots, V_d)$ with
\begin{equation}  \label{Vj}
   V_j = M_{U'_{j*} P^{\prime \perp}_{j*} + z U_{j*} P'_{j*}}\oplus W_{\partial j} \text{ for } j = 1, \dots, d
\end{equation}
is a (not necessarily commutative) isometric lift for of $(T_1, \dots, T_d)$, where the construction involves only an
And\^o tuple $(\cF_*, \Lambda_{j*}, P_{j*}, U_{j*})_{j=1}^d$ for $\underline{T} = (T_1, \dots, T_d)$.   Note also that the presentation
\eqref{Vj} shows that $(V_1, \dots, V_d)$ is just the (not necessarily commutative) BCL-model associated with the BCL-tuple
$(\cF_*, \cR_D, P_{*j}, U_{*j}, W_{\partial j})_{j=1}^d$ as in Definition \ref{D:BCLtuple}.
\end{remark}

\section{Pseudo-commutative contractive lifts and models for tuples of commutative contractions}  \label{S:pcc}
One disadvantage of dilation theory in $\overline{\mathbb{D}}^d$ ($d\geq2$) is that there is no uniqueness of minimal isometric lifts when such
exist, even in the case $d=2$ (where at least we know such exist)---unlike the classical case $d=1$.  We next identify an alternative generalization
of the notion of isometric lift for the $d=1$ case, which, as we shall see, always exists and has good uniqueness properties.

\begin{definition}\label{D:pcc}
For a given $d$-tuple $\underline{T}=(T_1,T_2,\dots,T_d)$ of commutative contractions acting on a Hilbert space $\cH$, we say that $(\Pi,\mathcal{K}, \underline{S},V)$ is a {\em pseudo-commutative contractive lift} of $(T_1,T_2,\dots,T_d,T)$, where $\underline{S}=(S_1,S_2,\dots,S_d)$ and
$T =T_1T_2\cdots T_d$, if
\begin{enumerate}
  \item $\Pi:\mathcal{H}\to\mathcal{K}$ is an isometry such that $(\Pi,\cK,V)$ is the minimal isometric lift of the single operator $T$, and
  \item with $S_j'=S_{j}^*V$ for each $1\leq j \leq d$, the pairs $(S_j,V)$, $(S_j',V)$ are commutative and $$(S_j^*,S_j'^*)\Pi=\Pi (T_j^*,T_{(j)}^*).$$
\end{enumerate}
\end{definition}

\begin{remark}\label{R:pcc}
 Note that we do not assume that the tuples $\underline{S}=(S_1,S_2,\dots,S_d)$ and $\underline{S}'=(S_1',S_2',\dots,S_d')$ be commutative but we do require that each of the pairs $(S_j,V)$ and $(S_j',V)$ be commutative. Also  one can show that the validity of the equation $S_j'=S_{j}^*V$
 implies the validity of the equation $S_j=S_j'^*V$ for each $j=1,2,\dots,d$ as follows:
\begin{align*}
S_j' = S_j^* V \Rightarrow S_j^{\prime *}  =  V^* S_j \Rightarrow  S_j^{\prime *} V = V^* S_j V = V^* V S_j  = S_j.
\end{align*}

Suppose that $(V_1,V_2,\dots,V_d)$ on $\cK$ is a commutative isometric lift of a given $d$ tuple of commutative contractions $\underline{T}=(T_1,T_2,\dots,T_d)$ on $\cH$ via an isometric embedding $\Pi:\cH\to\cK$. Let us denote by $V$ the isometry $V_1V_2\cdots V_d$. Then with $V_j':=V_1\cdots V_{j-1}V_{j+1}\cdots V_d$ for $j=1,2,\dots,d$, we see that part (2) in Definition \ref{D:pcc} is satisfied. However, the lift $(\Pi,\cK,V)$ of $T=T_1T_2\cdots,T_d$ need not be minimal, i.e. condition (1) in Definition \ref{D:pcc} may not hold.
\end{remark}

The next theorem shows that for a given tuple $\underline{T}$ of commutative contractions, a pseudo-commutative contractive lift exists and any two such lifts are unitarily equivalent in a sense explained in the theorem.

\begin{thm}\label{Thm:U-T}
Let $\underline{T}=(T_1,T_2,\dots,T_d)$ be a $d$-tuple of commutative  contractions acting on a Hilbert space $\cH$ and
 let $T =T_1T_2\cdots T_d$.
Then there exists a pseudo-commutative contractive lift of $(T_1,T_2,\dots,T_d,T)$. Moreover,
if $(\Pi_1,\mathcal{K}_1, \underline{S},V_1)$ and
$(\Pi_2,\mathcal{K}_2, \underline{R},V_2)$ are two pseudo-commutative contractive lifts of $(T_1,T_2,\dots,T_d,T)$,
where $\underline{S}=(S_1,S_2,\dots,S_d)$ and $\underline{R}=(R_1,R_2,\dots,R_d)$, then $(\Pi_1,\mathcal{K}_1, \underline{S},V_1)$
and $(\Pi_2,\mathcal{K}_2, \underline{R},V_2)$ are unitarily equivalent in the sense of Definition \ref{D:terminology}.
\end{thm}

\begin{proof}[Proof of Existence]
Roughly the idea is that a pseudo-commutative contractive lift of $\underline{T} = (T_1, \dots, T_d)$ arises as the compression of the And\^o-tuple-based
noncommutative isometric lift constructed in Theorem \ref{Thm:noncomIsolift} to the minimal lift space for the single contraction operator
$T = T_1 \cdots T_d$.  Precise details area as follows.

We use the Douglas model for the minimal isometric lift of the single contraction operator $T = T_1 \cdots, T_d$ as described in Section
 \ref{S:Douglas}, namely
 $$
  (H^2(\cD_{T^*}) \oplus \cR_D, \Pi_D, V_D = M_z \oplus W_D)
 $$
 as in \eqref{Pi-D}, \eqref{DougInt}, \eqref{Vd}.  We let $(W_{\partial 1}, \dots, W_{\partial d})$ be the canonical unitary tuple associated with
 $(T_1, \dots, T_d)$ as in \eqref{canonicalWpartial}, and let $\{G_{i1}, G_{i2} \colon i = 1, \dots d\}$ be the fundamental operators associated with
 $\underline{T}^* = (T_1^*, \dots, T_d^*)$ as in \eqref{En-fundeqn*}.
 Set
 \begin{align}\label{Dmodelops}
\nonumber\underline{S}^{D}&:=(S_1^{D},S_2^{D},\dots,S_d^{D})\\&:=(M_{G_{11}^*+zG_{12}}\oplus
W_{\partial 1},M_{G_{21}^*+zG_{22}}\oplus W_{\partial 2},\dots,M_{G_{d1}^*+zG_{d2}}\oplus W_{\partial d})
\end{align}
and
\begin{align}\label{Dmodelops'}
\nonumber\underline{S'}^{D}&:=(S_1'^{D},S_2'^{D},\dots,S_d'^{D})\\&:=(M_{G_{12}^*+zG_{11}}\oplus W_{(\partial 1)},M_{G_{22}^*+zG_{21}}
\oplus W_{(\partial 2)}\dots,M_{G_{d2}^*+zG_{d1}}\oplus W_{(\partial d)}).
\end{align}
We shall show that
\begin{equation}  \label{Dpcc}
( \Pi_D, \cK_D = H^2(\cD_T) \oplus \cR_D, \underline{S}^D, M_z \oplus W_D)
\end{equation}
is a pseudo-commutative contractive lift of $\bT = (T_1, \dots, T_d)$.

 Toward this goal let us first note that part (1) of Theorem \ref{Thm:HalmosDil} for the $d$ tuple $\underline{T}^*$, we get
\begin{align}\label{GsHalmosDil}
(G_{j1},G_{j2})=\Lambda_{j*}^*(P_{j*}^\perp U_{j*}^*,U_{j*}P_{j*})\Lambda_{j*} \text{ for each }j=1,2,\dots, d,
\end{align}
where $(\cF_{*},\Lambda_{j*},P_{j*},U_{j*})_{j=1}^d$ is an And\^o tuple for $\underline{T}^*$.
We now recall the construction of a noncommutative isometric lift described in Theorem \ref{Thm:noncomIsolift}. Notice that the isometries
$\Pi_{j*}:\cH\to H^2(\cF_*)\oplus\cR_D$
as in (\ref{Pi-j*}) can be factored as
\begin{equation} \label{Pij*}
\Pi_{j*}h=\begin{bmatrix}
            (I_{H^2}\otimes \Lambda_{j*}) & 0 \\
            0 & I_{\cR_D}
          \end{bmatrix}\begin{bmatrix}
                         \widehat \cO_{D_{T^*}, T^*}(z) h \\
                         Qh
                       \end{bmatrix}=\begin{bmatrix}
            (I_{H^2}\otimes \Lambda_{j*}) & 0 \\
            0 & I_{\cR_D}
          \end{bmatrix}\Pi_Dh.
\end{equation}
Therefore from the first equation in (\ref{AbstrctDil}) we get for each $j=1,2,\dots,d$,
\begin{align*}
\Pi_D T_j^*h&=\begin{bmatrix}
            (I_{H^2}\otimes \Lambda_{j*}^*) & 0 \\
            0 & I_{\cR_D}
          \end{bmatrix}\begin{bmatrix}
                           M_{U_{j*}P_{j*}^\perp+zU_{j*}P_{j*}}^* & 0 \\
                           0 & W_{\partial j}^*
                         \end{bmatrix}
           \begin{bmatrix}
            (I_{H^2}\otimes \Lambda_{j*}) & 0 \\
            0 & I_{\cR_D}
          \end{bmatrix}\Pi_Dh\\
&=\begin{bmatrix}
     M_{\Lambda_{j*}^*U_{j*}P_{j*}^\perp\Lambda_{j*}+z\Lambda_{j*}^*U_{j*}P_{j*}\Lambda_{j*}}^* & 0 \\
    0 &  W_{\partial j}^*
  \end{bmatrix}\Pi_Dh\\
&=\begin{bmatrix}
    M_{G_{j1}^*+zG_{j2}}^* & 0 \\
    0 & W_{\partial j}^*
  \end{bmatrix}\Pi_Dh.
\end{align*}
Consequently, we have for each $j=1,2,\dots, d$,
\begin{align}\label{semiDdil}
& \Pi_DT_j^*=(M_{G_{j1}^*+zG_{j2}}^*\oplus W_{\partial j}^*)\Pi_D = S_j^{D *} \Pi_D h
\end{align}
Similarly starting with the second equation in (\ref{AbstrctDil}) and proceeding as above we obtain
\begin{align}\label{semiDdil1}
\Pi_DT_{(i)}^*=(M_{G_{i2}^*+zG_{i1}}^*\oplus W_{(\partial i)}^*)\Pi_D = S^{\prime D *}_i \Pi_D.
\end{align}
Then with $V_D=M_z\oplus W_D$, the Douglas isometric lift of $T=T_1T_2\cdots T_d$ as discussed in Section \ref{S:Douglas}, it follows
from the equality (see (\ref{prodUni}))
$$
W_D=W_{\partial 1}W_{\partial 2}\cdots W_{\partial d}
$$
that $S_j'^{D}=S_j^{D *}V_{D}$ for each $j=1,2,\dots, d$. As we have already noted,
$(\Pi,\cK_D,V_D)$ is a minimal isometric lift of $T$. Therefore part (1) of Definition \ref{D:pcc} is satisfied.
Also it follows from definitions (\ref{Dmodelops}) and (\ref{Dmodelops'}) that the pairs $(S_j,V_D)$ and $(S_j',V_D)$ are commutative
for each $j=1,2,\dots,d$. And finally from equations (\ref{semiDdil}) and (\ref{semiDdil1}) we see that part (2) of Definition \ref{D:pcc} is also
satisfied. Consequently,  \eqref{Dpcc} is a pseudo-commutative contractive lift of $(T_1,T_2,\dots,T_d,T)$.
\end{proof}

\begin{proof}[Proof of Uniqueness in Theorem \ref{Thm:U-T}]
The strategy is to show that any pseu\-do-commut\-ative contractive lift $(\Pi,\mathcal{K}, \underline{S},V)$
is unitarily equivalent to
the canoni\-cal-model pseudo-commutative contractive lift $(\Pi_D,\mathcal{K}_D, \underline{S}^D,V_D)$,
as constructed in the existence part
of the proof, Since $(\Pi,V)$ and $(\Pi_{D},V_{D})$ are two minimal
isometric dilations of $T=T_1T_2\cdots T_d$, there exists a unitary $\tau :\cK\to\cK_{D}$ such that
$\tau V=V_{D}\tau  \text{ and }\tau \Pi=\Pi_{D}.$
We show that this unitary does the rest of the job.

Without loss of generality we may assume that $(\Pi,V)=(\Pi_{D},V_{D})$. Due to this reduction all we have to show is that $\underline{S}=\underline{S}^{D}$ and $\underline{S}'=\underline{S}'^{D}$. First let us suppose
\begin{align}\label{Sj&'}
 S_j=\begin{bmatrix}
       A_j & B_j  \\
       C_j & D_j
      \end{bmatrix} \text{ and }
      S_j'=\begin{bmatrix}
       A_j' & B_j'  \\
       C_j' & D_j'
      \end{bmatrix}
\end{align}
for each $j=1,2,\dots, d$ with respect to the decomposition $\cK_{D}=H^2(\cD_{T^*})\oplus\cR_D$.
Since each $S_j$ commutes with $V_D=M_z\oplus W_D$, we have
\begin{align}\label{intpure&uni}
\nonumber&\begin{bmatrix}
       A_j & B_j  \\
       C_j & D_j
      \end{bmatrix}\begin{bmatrix}
       M_z & 0  \\
       0 & W_D
      \end{bmatrix} = \begin{bmatrix}
       M_z & 0  \\
       0 & W_D
      \end{bmatrix} \begin{bmatrix}
       A_j & B_j  \\
       C_j & D_j
      \end{bmatrix} \\
      \Leftrightarrow&
      \begin{bmatrix}
       A_jM_z & B_j W_D\\
       C_jM_z & D_jW_D
      \end{bmatrix} =
      \begin{bmatrix}
       M_zA_j & M_zB_j  \\
       W_DC_j & W_DD_j
      \end{bmatrix}.
\end{align}
It is well-known that any operator that intertwines a unitary and a pure isometry is zero
(see e.g.\  \cite[page 227]{NFintertwine} or \cite[Chapter 3]{BS-Memoir}),
hence $B_j = 0$ for $j=1,2,\dots, d$.
Since each $S_j'$ commutes with $V$ also, by a similar computation with $S_j'$, we have $B_j'=0$ for each $j=1,2,\dots, d$.
From the identity of (1,1)-entries in \eqref{intpure&uni} we see that
\begin{equation}   \label{multiplier}
A_j=M_{\varphi_j} \text{ and }A_j'=M_{\varphi_j'}, \text{ for some }\varphi_j,\varphi_j'\in H^\infty(\cB(\cD_{T^*})).
\end{equation}
Hence $S_j$ and $S_j'$ have the form
\begin{equation}   \label{Sjs}
S_j = \begin{bmatrix} M_{\varphi_j} & 0 \\ C_j & D_j \end{bmatrix}, \quad S'_j = \begin{bmatrix}  M_{\varphi'_j} & 0 \\ C'_j & D'_j \end{bmatrix}.
\end{equation}
Since $S_j'=S_j^*V$ and hence also by Remark \ref{R:pcc} $S_j=S_j'^*V$, we then have
\begin{align}
 & \begin{bmatrix} M_{\varphi'_j} & 0  \\ 0 & D'_j\end{bmatrix}=\begin{bmatrix} M_{\varphi_j}^* & C_j^* \\ 0 & D_j^* \end{bmatrix}
      \begin{bmatrix} M_z & 0  \\  0 & W_D \end{bmatrix} =
      \begin{bmatrix} M_{\varphi_j}^*M_z & C_j^* W_D  \\ 0 & D_j^*W_D
      \end{bmatrix},  \notag   \\
 &   \begin{bmatrix}  M_{\varphi_j} & 0  \\ 0 & D_j\end{bmatrix}=\begin{bmatrix} M_{\varphi_j'}^* & C_j^{\prime *}   \\ 0 & D_j'^* \end{bmatrix}
      \begin{bmatrix} M_z & 0  \\  0 & W_D \end{bmatrix} =
      \begin{bmatrix} M_{\varphi_j'}^*M_z & C_j'^* W_D  \\ 0 & D_j'^*W_D
      \end{bmatrix}.
      \label{Sj&*}
\end{align}
From equality of the (1,2) entries we see that
$$
0 = C_j^* W_D, \quad 0 = C_j'^* W_D.
$$
As $W_D$ is unitary, in particular $W_D$ is surjective and we may conclude that in fact $C_j = 0$, $C_j'^* = 0$ and the form \eqref{Sjs}
for $S_j$ and $S_j'$ collapses to
\begin{equation} \label{Sjs'}
S_j = \begin{bmatrix} M_{\varphi_j} & 0 \\ 0 & D_j \end{bmatrix}, \quad S'_j = \begin{bmatrix}  M_{\varphi'_j} & 0 \\ 0 & D'_j \end{bmatrix}.
\end{equation}
Looking next at the identities $M_{\varphi_j'}=M_{\varphi_j}^*M_z$ and $M_{\varphi_j} = M_{\varphi'_j}^* M_z$ for each $j=1,2,\dots, d$
in terms of power series expansions of $\varphi_j$ and $\varphi_j'$ then leads to
\begin{equation}  \label{pencils}
 \varphi_j(z)=\widetilde G_{j1}^*+z\widetilde G_{j2} \text{ and }  \varphi_j'(z)=\widetilde G_{j2}^*+z\widetilde G_{j1} \text{ for } j=1,2,\dots, d
\end{equation}
for some $\widetilde G_{j1},\widetilde G_{j2}\in\cB(\cD_{T^*})$.  We shall eventually see that
$\{\widetilde G_{j1},\widetilde G_{j2} \colon j=1,2,\dots,d\}$ is exactly the set of fundamental operators $\{ G_{j1}, G_{j2} \colon j=1,2, \dots, d\}$
for $\underline{T}^*$.

Let us now analyze the second components in \eqref{Sjs'} involving the operator tuples $(D_1,D_2,\dots, D_d)$ and $(D_1',D_2',\dots, D_d')$.
From the relations
$$
(S_j^*,S_j'^*)\Pi_D=\Pi_D (T_j^*,T_{(j)}^*)
$$
we have for all $h\in\cH$,
$$
D_j^*Qh=QT_j^*h \text{ and }D_j'^*Qh=QT_{(j)}^*h,
$$
which by (\ref{TheXi}) implies that $D_j^*|_{\overline{\operatorname{Ran}}\,Q}=X_j^*$, for each $j=1,2,\dots,d$. Since each $S_j$ and
$S_j'$ commute with $V_D$, $D_j$ and $D_j'$ commute with $W_D$ and since $W_D$ is a unitary, $D_j^*$ and $D_j'^*$ also commute with
$W_D$. Using this we have for every $\xi\in\overline{\operatorname{Ran}}\,Q$ and $n\geq 0$,
\begin{align*}
&D_j^*(W_D^n)\xi=W_D^nD_j^*\xi=W_D^nX_j^*\xi=W_D^nW_j^*\xi=W_j^*W_d^n\xi\\
&D_j'^*(W_D^n)\xi=W_D^nD_j'^*\xi=W_D^nX_{(j)}^*\xi=W_D^nW_{(j)}^*\xi=W_{(j)}^*W_d^n\xi.
\end{align*}
As the set of elements of the form $W_D^n \xi$ is dense in $\cR_D$,  we conclude that
\begin{align}\label{Dconclusion}
  D_j=W_j \text{ and }D_j'=W_{(j)} \text{ for each }j=1,2,\dots,d.
\end{align}

To show that $G_{j1} = \widetilde G_{j1}$, by the uniqueness result in part (2) of Theorem \ref{fund-existence} it suffices to show that
\begin{equation}   \label{toshow}
D_{T^*} G_{j1} D_{T^*}  = D_{T^*} \widetilde{G}_{j1} D_{T^*}.
\end{equation}
for $j=1, \dots, d$.
The fundamental operator $G_{j1}$ is characterized as the unique solution of
\begin{equation}   \label{1}
D_{T^*} G_{j1} D_{T^*} = T_j^* - T_{(j)} T^*
\end{equation}
As $\underline{S}$ is a pseudo-commutative contractive lift of $\underline{T}$ with embedding operator $\Pi_D$,
we have by Definition \ref{D:pcc} the intertwining conditions
$$
S_j^*  \Pi_D = \Pi_D T_j^*, \quad S_j^{\prime *} \Pi_D = \Pi_D T_{(j)}^*, \quad W_D^* \Pi_D = \Pi_D T^*
$$
from which we also deduce that
$$
 T_{(j)} = T_{(j)} \Pi_D^* \Pi_D = \Pi_D^* S_j' \Pi_D.
$$
We may then compute
\begin{align}
& T_j^* - T_{(j)} T^*   = \Pi_D^* \Pi_D (T_j^* - T_{(j)} T^*) = \Pi_D^* S_j^* \Pi_D - \Pi_D^* \Pi_D \Pi_D^* S_j' \Pi_D T^* \notag \\
& \quad  = \Pi_D^* (S_j^* - S_j' W_D^*) \Pi_D = \Pi_D^* (S_j^* - S_j^* W_D W_D^*)   = \Pi_D^* S_j^* (I - W_D W_D^*)  \notag \\
& \quad = \Pi_D^* \begin{bmatrix} M_{\varphi_j}^* (I - M_z M_z^*) & 0 \\ 0 & 0 \end{bmatrix} =
\Pi_D^* \begin{bmatrix} (I - M_z M_z^*) \otimes \widetilde G_{j1} & 0 \\ 0 & 0 \end{bmatrix}  \notag \\
& \quad = \widehat \cO^*_{D_{T^*}, T^*} \left((I - M_z M_z^*) \otimes \widetilde G_{j1}\right).
\label{2}
\end{align}
From the general formula
$$
\widehat \cO^*_{D_{T^*}, T^*} \colon \sum_{n=0}^\infty h_n z^n \mapsto \sum_{n=0}^\infty T^n D_{T^*} h_n
$$
for the action of the adjoint observability operator $\cO^*_{D_{T^*}, T^*}$ and combining \eqref{1} and \eqref{2}, we finally arrive at
$$
 D_{T^*} G_{j1} D_{T^*} = \widehat \cO^*_{D_{T^*}, T^*} \left( (I - M_z M_z^*) \otimes \widehat G_{j1} \right) D_{T^*}
 = D_{T^*} \widetilde G_{j1} D_{T^*}
 $$
 and \eqref{toshow} follows as wanted.

A similar computation shows that
$$
D_T^* G_{j2} D_{T^*} = T_{(i)}^* - T_i T^* =  \cdots =D_{T^*} \widetilde G_{j2} D_{T^*}
$$
for $j = 1, \dots, d$ from which it follows that $\widetilde G_{j2} = G_{j2}$ as well.
This completes the proof of uniqueness in Theorem \ref{Thm:U-T}.
\end{proof}

\begin{remark}  \label{R:Dmodel}
The proof of the existence part of Theorem \ref{Thm:U-T} actually gives a canonical model \eqref{Dmodelops}-\eqref{Dmodelops'}
for an arbitrary pseudo-commutative contractive lift of a given commutative contractive operator-tuple $\underline{T} = (T_1, \dots, T_d)$.
By compressing the operators $\underline{S}^D$ to the subspace
$$
  \cH_D = (H^2(\cD_{T^*}) \oplus \cR_D) \ominus \operatorname{Ran} \Pi_D,
$$
we arrive at a Douglas-type functional model for the original commutative contractive operator tuple.
The precise statement is:
{\sl Let $\underline{T}=(T_1,T_2,\dots,T_d)$ be a $d$-tuple of commutative  contractions on a Hilbert space $\cH$ and $T=T_1T_2\cdots T_{d}$. Let $\{G_{i1},G_{i2}:i=1,\dots,d\}$ be the fundamental operators of the adjoint tuple $\underline{T}^*=(T_1^*,T_2^*,\dots,T_d^*)$ and $\underline{W}_{\partial}=(W_{\partial 1},W_{\partial 2},\dots ,W_{\partial d})$ be the canonical commutative unitary-operator tuple associated with $\underline{T}$
as in \eqref{canonicalWpartial}.
Then the tuple $(T_1,\dots,T_{d},T)$ is unitarily equivalent to
\begin{align}\label{Dmodel}
 P_{\mathcal{H}_D}(M_{G_{11}^*+zG_{12}}\oplus W_{\partial 1},\dots,M_{G_{d1}^*+zG_{d2}}\oplus W_{(\partial d)},M_z\oplus W_D)|_{\mathcal{H}_D},
\end{align} and $(T_{(1)},\dots,T_{(d)},T)$ is unitarily equivalent to
\begin{align}\label{Dmodel(i)}
 P_{\mathcal{H}_D}(M_{G_{12}^*+zG_{11}}\oplus W_{(1)},\dots,M_{G_{d2}^*+zG_{d1}}\oplus W_{(d)},M_z\oplus W_D)|_{\mathcal{H}_D},
\end{align}where $\mathcal{H}_D=(H^2(\cD_{T^*})\oplus\cR_D)\ominus\operatorname{Ran }\Pi_D$.}
\end{remark}

We saw in the above proof of uniqueness that the unitary involved in two pseudo-commutative contractive lifts
$(\Pi_1,\mathcal{K}_1, \underline{S},V_1)$ and $(\Pi_2,\mathcal{K}_2, \underline{R},V_2)$ is the same unitary that is involved
in the unitary equivalence of the two minimal isometric lifts $(\Pi_1,\cK_1,V_1)$ and $(\Pi_2,\cK_2,V_2)$ of $T$. Since such a unitary
is unique (see the proof of Theorem I.4.1
in \cite{Nagy-Foias}), we have the following consequence of Theorem \ref{Thm:U-T}.
\begin{corollary}\label{Uni-Cor}
Let $\underline{T}=(T_1,T_2,\dots,T_d)$ be a $d$-tuple of commutative  contractions acting on a Hilbert space and $T=T_1T_2\cdots T_d$.
If $(\Pi_1,\mathcal{K}_1, \underline{S},V_1)$ and $(\Pi_2,\mathcal{K}_2, \underline{R},V_2)$ be two pseudo-commutative contractive lifts of
$(T_1,T_2,\dots,T_d,$ $T)$ such that $(\Pi_1,\cK_1,V_1)=(\Pi_2,\cK_2,V_2)$, then $\underline{S}= \underline{R}.$
\end{corollary}

We end this section with another model for tuples of commutative contractive operator-tuples. This model will be used crucially in the next section where we analyze characteristic tuples for a given commutative contractive operator-tuple.

Sz.-Nagy and Foias gave a concrete functional model for the minimal isometric dilation for the case of a completely nonunitary (c.n.u.)
 contraction (see \cite{Nagy-Foias} for
a comprehensive treatment).  In their construction of this functional
model appears what they called the {\em{characteristic function}} for a contraction operator $T$ on a Hilbert space $\cH$,
a contractive analytic function on the unit disk $\mathbb D$ defined explicitly in terms of $T$ via the formula
\begin{equation}   \label{CharcFunc}
\Theta_T(z):=[-T+zD_{T^*}(I_\mathcal{H}-zT^*)^{-1}D_T]|_{\mathcal{D}_T} \colon \cD_T \to \cD_{T^*}
\text{ for  $z \in \mathbb{D}$}.
\end{equation}
Also key to their analysis is the so-called {\em{defect}} of the characteristic function  $\Delta_T$ defined a.e. on the unit circle ${\mathbb T}$ as
\begin{equation}    \label{DefectCharc}
\Delta_{T}(\zeta) := (I - \Theta_T(\zeta)^* \Theta_T(\zeta))^{1/2},
\end{equation}
where $\Theta_T(\zeta)$ is the radial limit of the characteristic function.
There it is shown that $(\Pi_{NF},V_{NF})$ is a minimal isometric dilation of $T$, where $V_{NF}$ is the isometry
\begin{equation}\label{Vnf}
 V_{NF}:=M_z\oplus M_{\zeta}|_{\overline{\Delta_T L^2(\cD_T)}} \text{ on }
 \cK_{NF}:=H^2(\cD_{T^*})\oplus\overline{\Delta_T L^2(\cD_T)}
\end{equation}
and $\Pi_{NF}:\cH\to \cK_{NF}$ is some isometry with
\begin{equation}\label{RanPinf}
\cH_{NF}:=\operatorname{Ran}\Pi_{NF}=\begin{bmatrix} H^2(\cD_{T^*}) \\ \overline{ \Delta_{T}L^2(\cD_T)} \end{bmatrix}
  \ominus \begin{bmatrix} \Theta_T \\ \Delta_{T} \end{bmatrix} \cdot H^2(\cD_T).
\end{equation}
It is shown in \cite{BS-Memoir} that, in case $T$ is completely nonunitary, the isometric embedding $\Pi_{NF}$
has the explicit formula
\begin{equation}\label{Pinf}
\Pi_{NF}=(I_{H^2\otimes \cD_{T^*}}\oplus u_{\text{min}})\Pi_D,
\end{equation}
where $u_{\text{min}}:\cR_D\to\overline{\Delta_T L^2(\cD_T)}$ is a unitary that intertwines $W_D$ and
$M_{\zeta}|_{\overline{\Delta_T L^2(\cD_T)}}$. Let us introduce the notation
\begin{align}\label{sharpWs}
  W_{\sharp j}:=u_{\text{min}}W_{\partial j}u_{\text{min}}^*, \quad  U_{\text{min}}:=((I_{H^2}\otimes
  \cD_{T^*})\oplus u_{\text{min}}
\end{align}
for unitary operators $W_{\sharp j}$ on $\overline{ \Delta_{T}L^2(\cD_T)}$ for $j=1, \dots, d$ and a unitary operator
$U_{\rm min} \colon\cR_D \to \overline{ \Delta_{T}L^2(\cD_T)}$.   Then we have
$$
U_{\text{min}}V_D=V_{NF}U_{\text{min}} \text{ and }U_{\text{min}}\Pi_D=\Pi_{NF}.
$$
Using this relation between $\Pi_D$ and $\Pi_{NF}$ we have the following intertwining relations that follow from (\ref{semiDdil}) and (\ref{semiDdil1}), respectively.
\begin{equation} \label{semiNFdil}
 \Pi_{NF}T_i^*=(M_{G_{i1}^*+zG_{i2}}^*\oplus W_{\sharp i}^*)\Pi_{NF}, \quad  \Pi_{NF}T_{(i)}^*=(M_{G_{i2}^*+zG_{i1}}^*\oplus W_{(\sharp i)}^*)\Pi_{NF}.
\end{equation}
Equations (\ref{semiNFdil}) then provide us a Sz.-Nagy--Foias type functional model for tuples of commutative contractions.

\begin{thm}\label{Thm:NFmodel}
  Let $\underline{T}=(T_1,T_2,\dots,T_d)$ be a $d$-tuple of commutative  contractions on a Hilbert space $\cH$ such that
  the contraction operator $T=T_1T_2\cdots T_{d}$ is c.n.u. Let $\{G_{i1},G_{i2}:i=1,\dots,d\}$ be the
  fundamental operators of the adjoint tuple $\underline{T}^*=(T_1^*,T_2^*,\dots,T_d^*)$
  and let the model space $\cH_{NF}$ be as in \eqref{RanPinf}. Then $(T_1,\dots,T_{d},T)$ is unitarily equivalent to
\begin{align}\label{NFmodel}
 P_{\mathcal{H}_{NF}}(M_{G_{11}^*+zG_{12}}\oplus W_{\sharp 1},\dots,M_{G_{d1}^*+zG_{d2}}\oplus W_{\sharp d},M_z\oplus M_{\zeta}|_{\overline{\Delta_T L^2(\cD_T)}})|_{\mathcal{H}_{NF}},
\end{align}
and $(T_{(1)},\dots,T_{(d)},T)$ is unitarily equivalent to
\begin{align}\label{NFmodel(i)}
 P_{\mathcal{H}_{NF}}(M_{G_{12}^*+zG_{11}}\oplus W_{(\sharp1)},\dots,M_{G_{d2}^*+zG_{d1}}\oplus W_{(\sharp d)},M_z\oplus M_{\zeta}|_{\overline{\Delta_T L^2(\cD_T)}})|_{\mathcal{H}_{NF}}.
\end{align}
\end{thm}

\begin{remark}  \label{R:NFpcc}
Equations (\ref{semiNFdil}) also provide us another model for pseudo-commutative contractive lifts, at least for the
case where the $T = T_1 \cdots T_d$ is c.n.u.  Indeed,
let $\underline{T}=(T_1,T_2,\dots,T_d)$ be a $d$-tuple of commutative  contractions on a Hilbert space
$\cH$, $T=T_1T_2\cdots T_{d}$
and let $\{G_{i1},G_{i2}:i=1,\dots,d\}$ be the set of fundamental operators of the adjoint tuple $\underline{T}^*=(T_1^*,T_2^*,\dots,T_d^*)$.
Let us set
\begin{align}
& \underline{S}^{NF}:=(S_1^{NF},S_2^{NF},\dots,S_d^{NF})   \notag \\
&:=(M_{G_{11}^*+zG_{12}}\oplus W_{\sharp 1},M_{G_{21}^*+zG_{22}}\oplus W_{\sharp 2},\dots,M_{G_{d1}^*+zG_{d2}}\oplus W_{\sharp d})   \label{NFmodelops}  \\
& \underline{S'}^{NF}:=(S_1'^{NF},S_2'^{NF},\dots,S_d'^{NF}), \notag \\
&:=(M_{G_{12}^*+zG_{11}}\oplus W_{(\sharp 1)},M_{G_{22}^*+zG_{21}}
\oplus W_{(\sharp 2)}\dots,M_{G_{d2}^*+zG_{d1}}\oplus W_{(\sharp d)}).
 \label{NFmodelops'}
\end{align}
Then it follows from the definition and from $M_{\zeta}|_{\overline{\Delta_T L^2(\cD_T)}}=W_{\sharp 1}W_{\sharp 2}\cdots W_{\sharp d}$
that $S_j'^{NF}=S_j^{NF *}V_{NF}$ for each $j=1,2,\dots, d$, where $V_{NF}$ is the minimal isometric lift of $T$ as defined in (\ref{Vnf}). Hence
by Equations (\ref{semiNFdil}) it follows that {\sl the tuple $(\Pi_{NF},\mathcal{K}_{NF}, \underline{S}^{NF},V_{NF})$ is a pseudo-commutative
contractive lift of $(T_1,T_2,\dots,T_d,T)$.}
\end{remark}

\section{Characteristic triple for a tuple of commutative  contractions}   \label{S:char-triple}
Let $\underline{T}=(T_1,T_2,\dots,T_d)$ be a $d$-tuple of commutative  contractions on a Hilbert space $\cH$ and
$(\cF_{*},\Lambda_{j*},P_{j*},U_{j*})_{j=1}^d$ be an And\^o tuple for $\underline{T}^*=(T_1^*,T_2^*,$ $\dots, T_d^*)$. Let
$\{G_{j1},G_{j2}:j=1,2,\dots,d\}$ be the set of fundamental operators of $\underline{T}^*$. Then note that by part (2) of
Theorem \ref{Thm:HalmosDil} we have
\begin{align}\label{JHalmosDil*}
(G_{j1},G_{j2})=\Lambda_{1*}^*(\tau_{j*} P_{j*}^\perp U_{j*}^*\tau_{j*}^*,\tau_{j*}U_{j*}P_{j*}\tau_{j*}^*)\Lambda_{1*}
\text{ for } j=1,2,\dots, d.
\end{align}

\begin{definition}
Let $\underline{T}=(T_1,T_2,\dots,T_d)$ be a $d$-tuple of commutative  contractions on a Hilbert space, $\mathbb G_\sharp:=\{G_{j1},G_{j2}:j=1,2,\dots,d\}$ be the set of fundamental operators of $\underline{T}^*$ and $\mathbb W_\sharp:=(W_{\sharp 1},
W_{\sharp 2},\dots,W_{\sharp d})$ be the tuple of commutative  unitaries as in (\ref{sharpWs}). The triple $(\mathbb{G}_\sharp,
\mathbb{W}_\sharp,\Theta_T)$ is called the characteristic triple for $\underline{T}$, where $\Theta_T$ is the characteristic function
for the contraction $T=T_1T_2\cdots T_d$.
\end{definition}

Note that the expression (\ref{JHalmosDil*}) of the fundamental operators of $\underline{T}^*$ indicates the
dependence of the
characteristic triple on a choice of an And\^o tuple for $\underline{T}^*$. However, the uniqueness part of Theorem \ref{fund-existence} says
that the fundamental operators are uniquely determined by $\underline{T}$. Consequently, the characteristic triple,
despite  its apparent
dependence on a choice of an And\^o tuple, turns out to be uniquely determined already by the $d$-tuple $\underline{T}$
of commutative  contractions.
In fact, as Theorem \ref{Thm:CompUniInv} below explains, the characteristic triple (up to the natural notion of equivalence
to be defined next) is a complete unitary invariant for tuples of commutative contractions.

\begin{definition}  \label{D:coincide}
Let $(\mathcal{D},\mathcal{D}_*,\Theta)$, $(\mathcal{D'},\mathcal{D'_*},\Theta')$ be two purely contractive
analytic functions. Let $\mathbb{G}=\{G_{j1},G_{j2}:j=1,2,\dots,d\}$ on $\mathcal{D}_*$, $\mathbb{G}'=\{G_{j1}',G_{j2}':j=1,2,\dots,d\}$ on $\mathcal{D'_*}$
be two sets of contraction operators and $\mathbb{W}=(W_1,W_2,\dots, W_d)$ on $\overline{\Delta_\Theta L^2(\mathcal{D})}$,
$\mathbb{W}'=(W_1',W_2',\dots,W_d')$ on $\overline{\Delta_{\Theta'} L^2(\mathcal{D'})}$ be two tuples of commutative
unitaries such that their product is $M_{\zeta}$ on the respective spaces. We say that the two triples $(\mathbb{G},\mathbb{W},\Theta)$
and $(\mathbb{G}',\mathbb{W}',\Theta')$ {\em coincide} if:
\begin{itemize}
  \item[(i)] $(\mathcal{D},\mathcal{D_*},\Theta)$ and $(\mathcal{D'},\mathcal{D'_*},\Theta')$ {\em coincide}, i.e.,
there exist unitary operators $u: \mathcal{D} \to \mathcal{D'}$ and $u_{*}: \mathcal{D}_{*} \to \mathcal{D'}_{*}$
such that the diagram
\begin{align}\label{coindiagram}
\begin{CD}
\mathcal{D}_T @>\Theta(z)>> \mathcal{D}_{T^*}\\
@Vu VV @VVu_{*} V\\
\mathcal{D}_{T'} @>>\Theta'(z)> \mathcal{D}_{{T'}^*}
\end{CD}
\end{align}
commutes for each $z \in {\mathbb D}$.
\item[(ii)] The same unitary operators $u$, $u_*$ as in part (i)  satisfy the  additional intertwining conditions:
\begin{align*}
&  \mathbb{G}'=(G_1',G_2')=u_*\mathbb{G}u_*^*=(u_*G_1u_*^*,u_*G_2u_*^*,\dots,u_*G_du_*^*),  \notag \\
& \mathbb{W}'=(W_1',W_2')=\omega_u\mathbb{W}\omega_u^*=(\omega_uW_1\omega_u^*,\omega_uW_2\omega_u^*,\dots,
\omega_uW_d\omega_u^*),
\end{align*}
where $\omega_u:\overline{\Delta_{\Theta} L^2(\mathcal{D})}\to\overline{\Delta_{\Theta'} L^2(\mathcal{D'})}$
is the unitary map induced by $u$ according to the formula
\begin{equation}      \label{omega-u}
\omega_u:=(I_{L^2}\otimes u)|_{\overline{\Delta_{\Theta} L^2(\mathcal{D})}}.
\end{equation}
\end{itemize}
\end{definition}

\begin{thm}\label{Thm:CompUniInv}
Let $\underline{T}=(T_1,T_2,\dots,T_d)$ on $\cH$ and $\underline{T}'=(T_1',T_2',\dots,T_d')$ on $\cH'$ be two tuples of commutative  contractions.
Let $(\mathbb{G}_\sharp,\mathbb{W}_\sharp,\Theta_T)$ and $(\mathbb{G}'_\sharp,\mathbb{W}'_\sharp,\Theta_{T'})$ be the characteristic triples of
$\underline{T}$ and $\underline{T}'$, respectively, where $T=T_1T_2\dots, T_d$ and $T'=T_1'T_2'\cdots,T_d'$. If $\underline{T}$ and
$\underline{T}'$ are unitarily equivalent, then $(\mathbb{G}_\sharp,\mathbb{W}_\sharp,\Theta_T)$ and
$(\mathbb{G}'_\sharp,\mathbb{W}'_\sharp,\Theta_{T'})$ coincide.

Conversely, suppose in addition that $T$ and $T'$ are c.n.u.\ with characteristic triples
$(\mathbb{G}_\sharp,\mathbb{W}_\sharp,\Theta_T)$ and $(\mathbb{G}'_\sharp,\mathbb{W}'_\sharp,\Theta_{T'})$  coinciding.
Then $\underline{T}$ and $\underline{T}'$ are unitarily equivalent.
\end{thm}

\begin{proof}
First let us suppose that $\underline{T}$ and $\underline{T}'$ be unitarily equivalent via a unitary similarity $U:\cH\to\cH'$. Then
\begin{align}\label{IntwinDefects}
U(I-T^*T)=(I-T'^*T')U \text{ and }U(I-TT^*)=(I-T'T'^*)U
\end{align}
and the functional calculus for positive operators implies that $U$ induces two unitary operators
\begin{align}\label{u&u*}
u:=U|_{\cD_T}:\cD_T\to\cD_{T'} \text{ and }u_*:=U|_{\cD_{T^*}}:\cD_{T^*}\to\cD_{T'^*}.
\end{align}
A consequence of the Sz.-Nagy--Foias theory \cite{Nagy-Foias} is
that $u_*\Theta_{T}u^*=\Theta_{T'}$ showing $\Theta_T$ and $\Theta_{T'}$ coincide, i.e., condition (i) holds.

As for condition (ii),  note that since the fundamental operators satisfy the fundamental equations (\ref{En-fundeqn*}),
one can easily deduce using (\ref{IntwinDefects}) that
\begin{align}\label{IntwinFunds}
u_*(G_{j1},G_{j2})=(G_{j1}',G_{j2}') u_*\text{ for each }j=1,2,\dots,d,
\end{align}
where $\mathbb{G}_\sharp=\{G_{j1},G_{j2}:j=1,2,\dots,d\}$ and $\mathbb{G}'_\sharp=\{G_{j1}',G_{j2}':j=1,2,\dots,d\}$.
It remains to establish the unitary equivalence of $\mathbb{W}_\sharp$ and $\mathbb{W}'_\sharp$ via $\omega_u=
(I_{L^2}\otimes u)|_{\overline{\Delta_{T} L^2(\mathcal{D}_T)}}$. To this end, we consider the tuple
$(\widetilde\Pi,\widetilde{\mathcal{K}}, \widetilde{\underline{S}}, \widetilde{\underline{S}'},\widetilde V)$, where $\widetilde \cK = \cK_{NF}$
and $\widetilde{V}=V_{NF}$ as in (\ref{Vnf}) and where
\begin{align}
 \widetilde{\Pi}:= & ((I_{H^2}\otimes u_*^*)\oplus \omega_u^*)\Pi_{NF}'U:\mathcal{H}\to H^2(\mathcal{D}_{T^*})
\oplus\overline{\Delta_{T}L^2(\mathcal{D}_{T})}   \notag \\
 \widetilde{\underline{S}}:= & (\widetilde S_1,\widetilde S_2,\dots,\widetilde S_d):=  \notag \\
&  (M_{G_{11}^*+zG_{12}}\oplus W''_{ 1},M_{G_{21}^*+zG_{22}}\oplus W''_{ 2},\dots,M_{G_{d1}^*+zG_{d2}}\oplus W''_{ d}) \notag \\
  \widetilde{\underline{S'}}:=& (\widetilde S_1',\widetilde S_2',\dots,\widetilde S_d'):= \notag \\
  & (M_{G_{12}^*+zG_{11}}\oplus W''_{( 1)},M_{G_{22}^*+zG_{21}}\oplus W''_{( 2)}\dots,M_{G_{d2}^*+zG_{d1}}\oplus W''_{( d)})   \label{AuxNotasn}
\end{align}
with
$$
\mathbb{W}'':=(W''_1,W''_2,\dots,W''_d):=\omega_u\mathbb{W}_\sharp\omega_u^*=(\omega_u W_{\sharp1}\omega_u^*,\omega_u W_{\sharp2}\omega_u^*,\dots,\omega_u W_{\sharp d}\omega_u^*).
$$
By tracing through the intertwining properties of the unitary identification maps $U$ and $\omega_u$, one can see that actually
$\widetilde \Pi = \Pi_{NF}$.
A further consequence of these intertwining properties is that the tuple
$$
  (\Pi_{NF}, \cK_{NF}, \underline{S}^{NF}, \underline{S}^{\prime NF}, V_{NF})
$$
as in \eqref{NFmodelops} and \eqref{NFmodelops} being a pseudo-commutative contractive lift of $\underline{T}$ implies that
$(\Pi, \cK_{NF}, \underline{\widetilde S}, \underline{\widetilde S}', V_{NF})$ is a pseudo-commutative contractive lift of $\underline{T}$ as well.
A direct application of Corollary \ref{Uni-Cor} then tells us that ${\mathbb W}'' = {\mathbb W}_\sharp$, i.e.,
$$
  \omega_u W_{\sharp j} \omega_u^* = W_{\sharp j} \text{ for } j = 1, \dots, d
$$
and condition (ii) in Definition \ref{D:coincide} is now verified as wanted.

Conversely, assume that the product operators $T$ and $T'$ are c.n.u.\ and $\underline{T}$ and $\underline{T}'$ have characteristic triples which
coincide.  By Theorem \ref{Thm:NFmodel} each of $\underline{T}$ and $\underline{T}'$ is unitarily equivalent to its respective
Sz.-Nagy--Foias functional model. It is now a straightforward exercise to see that the unitary identification maps  $u$ and $u_*$ in the
definition of the coincidence of the characteristic triples leads to a unitary identification of the model spaces $\cH_{NF}$ and $\cH'_{NF}$
which also implements a unitary similarity of the respective model operator tuples \eqref{NFmodel} and \eqref{NFmodel(i)} associated with
$\underline{T}$ and $\underline{T}'$.  This completes the proof of  Theorem \ref{Thm:CompUniInv}.
\end{proof}

We next introduce the
 notion of {\em admissible triple} for a collection $\{  {\mathbb G},  {\mathbb W}, \Theta \}$ of the same sort as appearing
 in Definition \ref{D:coincide} but which satisfies some additional conditions; the additional conditions correspond to
 what is needed to conclude that the triple arises as the characteristic triple for some contractive commutative tuple
 $\underline{T}$.

\begin{definition}   \label{D:admissible}
Suppose that $(\cD, \cD_*, \Theta)$ is a purely contractive analytic function, ${\mathbb G} = \{ G_{j1}, G_{j2} \colon j = 1, \dots, d\}$
is a collections of operators on $\cD_*$,  ${\mathbb W} = \{ W_1, \dots, W_d\}$ is a commutative $d$-tuple of unitary operators on
$\overline{\Delta_\Theta L^2(\cD)}$ such that:
\begin{enumerate}
\item Each $M_{G_{j1}^* + z G_{j2}}$ is a contraction operator on $H^2(\cD_*)$.
\item $W_1 \cdots W_d = M_\zeta|_{\overline{\Delta_\Theta L^2(\cD)}}$.
\item The space $\cQ_\Theta : =\sbm{ \Theta \\ \Delta_\Theta } H^2(\cD) \subset \sbm{H^2(\cD_*) \\ \overline{\Delta_\Theta L^2(\cD)}}$ is jointly invariant for the operator tuple $\left\{ \sbm{ M_{G_{j1}^* + z G_{j2}} & 0 \\ 0 & W_j } \colon j=1, \dots, d\right\}$.
\item With $\cK_\Theta = \sbm{H^2(\cD_*) \\ \overline{\Delta_\Theta L^2(\cD)} }$ and $\cH_\Theta = \cK_\Theta \ominus \cQ_\Theta$ and with
operators $T_j$ on $\cH_\Theta$ defined by
\begin{equation}  \label{adm-funcmodel}
   {\mathbf T}_j = P_{\cH(\Theta)} \begin{bmatrix} M_{G_{j1}^* + z G_{j2}} & 0 \\ 0 & W_j \end{bmatrix} \big |_{\cH(\Theta)} \text{ for } j = 1, \dots, d,
\end{equation}
the operator-tuple $({\mathbf T}_1, \dots, {\mathbf T}_d)$ is commutative with product (in any order) then given by
$$
   {\mathbf T}_1 \cdots {\mathbf T}_d = P_{\cH(\Theta)} \begin{bmatrix}      M_z & 0 \\ 0 & M_\zeta  \end{bmatrix} \big |_{\cH(\Theta)}.
$$
\end{enumerate}
Then we shall say that the collection $\{ {\mathbb G}, {\mathbb W}, \Theta\}$ is an {\em admissible triple} and that the commutative
contractive operator-tuple $\underline{\mathbf T} = ({\mathbf T}_1, \dots, {\mathbf T}_d)$ acting on the space $\cH(\Theta)$
\eqref{adm-funcmodel} is the {\em functional model} associated with the admissible triple $\{ {\mathbb G}, {\mathbb W}, \Theta\}$.
\end{definition}

Let us note that the functional model associated with an admissible triple $\{{\mathbb G}, {\mathbb W}, \Theta\}$
also displays a pseudo-commutative contractive lift for its functional-model commutative, contractive operator
tuple ${\mathbf T}$, namely:
\begin{align*}
& {\mathbf S} : = \left\{ \begin{bmatrix} M_{G_{j1}^* + z G_{j2}} & 0 \\ 0 & W_j \end{bmatrix}
\colon j=1, \dots, d \right\}, \quad
 V = \begin{bmatrix}  M_z & 0 \\ 0 & M_\zeta \end{bmatrix} \\
& {\mathbf S}': = \left\{ \begin{bmatrix} M_{G_{j2}^* + z G_{j1}} & 0 \\ 0 & W_{(j)} \end{bmatrix} \colon j= 1, \dots, d \right\}.
\end{align*}

Note also that it easily follows from the definitions that the characteristic triple for a commutative contractive
$d$-tuple $\underline{T}$ is an admissible triple.
Furthermore the functional model associated with the characteristic triple
$({\mathbb G}_\sharp, {\mathbb W}_\sharp, \Theta_T)$ is the same as the
functional model obtained by considering $({\mathbb G}_\sharp, {\mathbb W}_\sharp, \Theta_T)$ as an admissible
triple.  The content of Theorem \ref{Thm:NFmodel} is that any commutative contractive tuple $\underline{T}$ is
unitarily equivalent to its associated functional model $\underline{\mathbf T}$.

Our next goal is to indicate the reverse path:  how to go from an admissible triple to a characteristic triple
for some commutative contractive pair $\underline{T}$.  We state the result without proof.

\begin{thm} \label{Thm:reverse}   If $({\mathbb G}, {\mathbb W}, \Theta)$ is an admissible triple, then $({\mathbb G},
{\mathbb W}, \Theta)$ is a characteristic triple for some contractive operator tuple.  More precisely,
the admissible triple $({\mathbb G}, {\mathbb W}, \Theta)$  coincides with the
characteristic triple  $({\mathbb G}_\sharp, {\mathbb W}_\sharp, \Theta_{\mathbf T})$ of its
functional model.
\end{thm}

Since model theory and unitary classification for commuting tuples of unitary operators can be handled by spectral theory,
the importance of the next result is that the c.n.u.~restriction on $T = T_1 \cdots T_d$ appearing in Theorem \ref{Thm:NFmodel}
and Theorem \ref{Thm:CompUniInv} is not essential. This result for the case $d=1$ goes back to Sz.-Nagy-Foias \cite{Nagy-Foias}.

\begin{thm}\label{Thm:CanDec}
Let $\underline{T}=(T_1,T_2,\dots,T_d)$ be a commutative contractive operator-tuple acting on a Hilbert space $\cH$. Then there corresponds a decomposition of $\cH$ into the orthogonal sum of two subspaces reducing each $T_j$, $j=1,2,\dots,d$, say
$\cH=\cH_{u}\oplus\cH_{c}$, such that with
\begin{align}
&(T_{1u},T_{2u},\dots,T_{du})=(T_1,T_2,\dots,T_d)|_{\cH_u}, \notag \\
&(T_{1c},T_{2c},\dots,T_{dc})=(T_1,T_2,\dots,T_d)|_{\cH_c},
\label{CanDec}
\end{align}
$T_u=T_{1u}T_{2u}\cdots T_{du}$ is a unitary and $T_{c}=T_{1c}T_{2c}\cdots T_{dc}$ is a completely nonunitary contraction.
Moreover, then $T_u \oplus T_c$ with respect to $\cH=\cH_{u}\oplus\cH_{c}$ is the Sz.-Nagy--Foias canonical decomposition for the contraction operator $T=T_1 T_2\cdots T_d$.
\end{thm}

\begin{proof}
Let $\{F_{j1},F_{j2}: j=1,2,\dots,d\}$ be the set of fundamental operators of $\underline{T}$. Then by Theorem \ref{fund-existence}, for each $j=1,2,\dots,d$, we have
\begin{equation}   \label{CanFundEqn}
T_j - T_{(j)}^* T = D_{T} F_{j1 } D_{T} \text{ and }
T_{(j)} - T_{j}^* T = D_{T} F_{j2 } D_{T}.
\end{equation}
By part (1) of Theorem \ref{Thm:HalmosDil} each of $F_{j1}$ and $F_{j2}$ are contractions. Consequently, we have for every
$\omega$ and $\zeta$ in $\mathbb{T}$
\begin{align}\label{CrucIneq1}
I_{\cD_T}-\operatorname{Re}(\omega F_{j1})\geq 0 \text{ and }I_{\cD_T}-\operatorname{Re}(\zeta F_{j2})\geq 0.
\end{align}
Adding together the two inequalities \eqref{CrucIneq1} then gives
\begin{align}\label{CrucIneq}
2I_{\cD_T}-\operatorname{Re}(\omega F_{j1}+\zeta F_{j2})\geq 0 \text{ for all }\omega,\zeta\in\mathbb{T}.
\end{align}
Recall that the fundamental operators act on $\cD_T=\overline{\operatorname{Ran}}\;D_T$. Therefore inequality (\ref{CrucIneq}) is equivalent to
\begin{align*}
2D_T^2-\operatorname{Re}(\omega D_TF_{j1}D_T+\zeta D_TF_{j2}D_T)\geq 0, \text{ for all }\omega,\zeta \in\mathbb{T},
\end{align*}
By (\ref{CanFundEqn}) we see that this in turn is the same as
\begin{align}\label{FinalIneq}
2D_T^2-\operatorname{Re}(\omega (T_j-T_{(j)}^*T))-\operatorname{Re}(\zeta (T_{(j)}-T_j^*T))\geq 0, \text{ for all }\omega,\zeta \in\mathbb{T}.
\end{align}
Let
\begin{align}\label{CanT}
  T=\begin{bmatrix}
      T_u & 0 \\
      0 & T_{c}
    \end{bmatrix}:\cH_u\oplus\cH_c\to\cH_u\oplus\cH_{c}
\end{align}
be the canonical decomposition of $T$ into unitary piece $T_u$ and completely nonunitary piece $T_{c}$. We show below that each $T_j$ is block diagonal with respect to the decomposition $\cH=\cH_u\oplus\cH_{c}$. Toward this end, we first suppose that with respect to the decomposition
$\cH=\cH_u\oplus\cH_{c}$ for each $j=1,2,\dots d$ we have
\begin{align}\label{Tjs}
T_j=\begin{bmatrix}
      A_j & B_j \\
      C_j & D_j
    \end{bmatrix} \text{ and }
    T_{(j)}=\begin{bmatrix}
      E_j & K_j \\
      L_j & H_j
    \end{bmatrix}.
\end{align}
Apply (\ref{FinalIneq}) to obtain that for each $j=1,2,\dots,d$,
\begin{align}\label{CrucIneq2}
 \nonumber \begin{bmatrix}
      0 & 0 \\
      0 & 2D_{T_{c}}^2
    \end{bmatrix}&-
    \operatorname{Re}\left(\omega \begin{bmatrix}
      A_j-E_j^*T_u & B_j-L_j^*T_{c} \\
      C_j-K_j^*T_u & D_j-H_j^*T_{c}
    \end{bmatrix}\right)\\&-
    \operatorname{Re}\left(\zeta \begin{bmatrix}
      E_j-A_j^*T_u & K_j-C_j^*T_{c} \\
      L_j-B_j^*T_u & H_j-D_j^*T_{c}
    \end{bmatrix}\right)\geq 0, \text{ for all }\omega,\zeta \in\mathbb{T}.
    \end{align}
    In particular, the $(1,1)$-entry in this inequality must satisfy
    \begin{align}\label{(11)}
      {\mathbb P}^j_{11}(\omega,\zeta):= \operatorname{Re}(\omega(A_j-E_j^*T_u ))+\operatorname{Re}(\zeta(E_j-A_j^*T_u))\leq 0, \text{ for all }\omega,\zeta \in\mathbb{T},
    \end{align}which implies that
    \begin{align*}
    {\mathbb P}^j_{11}(\omega,1)+ {\mathbb P}^j_{11}(\omega,-1)&=2\operatorname{Re}(\omega(A_j-E_j^*T_u ))\leq0 \text{ and}\\
    {\mathbb P}^j_{11}(1,\zeta)+{\mathbb P}^j_{11}(-1,\zeta)&=2\operatorname{Re}(\zeta(E_j-A_j^*T_u))\leq0.
    \end{align*}
It is an elementary exercise to show that, if a bounded operator $X$ such that $Re(\zeta X)\leq 0$ for all $\zeta\in\mathbb T$, then $X=0$
(see e.g.\ Lemma 2.4 in \cite{B-P-SR}). We apply this fact to conclude that
    \begin{align}\label{CrucIneq3}
     A_j=E_j^*T_u \text{ and }E_j=A_j^*T_u \text{ for each }j=1,2,\dots,d.
    \end{align}
    This shows that the $(1,1)$-entry of the matrix on the left-hand side of (\ref{CrucIneq2}) is zero. Since the
    matrix is positive semi-definite, the $(1,2)$-entry (and hence also the $(2,1)$-entry) is also zero, i.e., for all
    $\omega,\zeta \in\mathbb{T}$
    \begin{align*}
    {\mathbb P}^j_{12}(\omega,\zeta):= \omega(B_j-L_j^*T_{c})+\bar{\omega}(C_j^*-T_u^*K_j)+\zeta(K_j-C_j^*T_{c})
    +\bar{\zeta}( L_j^*-T_u^*B_j)=0,
    \end{align*}which in particular implies that
    \begin{align*}
      {\mathbb P}^j(\omega):={\mathbb P}^j_{12}(\omega,1)+ {\mathbb P}^j_{12}(\omega,-1)=
      2\omega(B_j-L_j^*T_{c})+2\bar{\omega}(C_j^*-T_u^*K_j)=0
    \end{align*}
    for every $\omega\in\mathbb{T}$. This implies the first two of the following equations while the last two are
    obtained similarly:
    \begin{align}\label{Cruceq}
      B_j=L_j^*T_{c},\quad C_j^*=T_u^*K_j,\quad K_j=C_j^*T_{c}\quad\text{and}\quad L_j^*=T_u^*B_j.
    \end{align}
Commutativity of each $T_j$ with $T$ gives
 \begin{align}\label{Cruceq2}
A_jT_u=T_uA_j,\quad B_jT_{c}=T_uB_j,\quad C_jT_u=T_{c}C_j\quad \text{and}\quad T_{c}D_j=D_jT_{c},
 \end{align}while commutativity of $T_{(j)}$ with $T$ implies
\begin{align}\label{Cruceq4}
E_jT_u=T_uE_j,\quad K_jT_{c}=T_uK_j,\quad L_jT_u=T_{c}L_j\quad \text{and}\quad T_{c}H_j=H_jT_{c}.
 \end{align}
Using the last equation in (\ref{Cruceq}) and the third equation in (\ref{Cruceq4}) we get
 $$
 B_j^*T_u^2=L_jT_u=T_{c}L_j=T_{c}B_j^*T_u.
 $$
 As $T_u$ is unitary, this leads to
 \begin{equation}  \label{Cruceq3}
 B_j^*T_u=T_{c}B_j^*.
 \end{equation}
  Using the second equality in (\ref{Cruceq2}) together with \eqref{Cruceq3} leads to
 \begin{align*}
T_{c}T_{c}^*B_j^*=T_{c}B_j^*T_u^*=B_j^*=B_j^*T_u^*T_u=T_{c}^*B_j^*T_{u}=T_{c}^*T_{c}B_j^*,
 \end{align*}
 which implies that $T_{c}$ is unitary on $\overline{\operatorname{Ran}}\, B_j^*$ for  $j=1,2,\dots,d$. Since $T_{c}$ is completely nonunitary, each
 $B_j$ must be zero. By similar arguments one can show that $C_j=0$, for each $j=1,2,\dots,d$. This completes the proof.
 \end{proof}

 \begin{remark}  \label{R:can-decom-pair}
 We note that a proof of Theorem \ref{Thm:CanDec}  is given in \cite{BS-Memoir} for the pair case ($d=2$).
 It is of interest to note that
 the general case can be reduced to the pair case simply by applying the result for the pair case to the special pair
 $(T_j, T_{(j)})$ for each $j=1, \dots, d$.  Our proof on the other hand is a direct multivariable proof.
  \end{remark}

  \begin{remark}  \textbf{Examples and special cases:}  If we consider the special case with $\underline{T}
  = \underline{V}$ is a commutative tuple of isometries $\underline{V} = (V_1, \dots, V_d)$ with product operator
  $V = V_1 \cdots V_d$ c.n.u.\ (i.e., $V$ is a pure isometry or shift operator), then the associated characteristic function
  $\Theta_V$ is zero, and the model theory presented here amounts to the BCL-model for commuting isometries as in
  Theorem  \ref{Thm:BCL}.  In detail, the characteristic triple collapses to the first component ${\mathbb G}$
  which has the additional structure of the form
  $$  G_{j1} = P_j^\perp U_j^*, \quad G_{j2} = U_j P_j$$
  for a collection of projection operators $P_j$ and unitary operators $U_j$ on a space $\cF$ ($j=1, \dots, d$)
  forming a BCL-tuple (Definition \ref{D:BCLtuple}) for which the associated isometric operator-tuple  $V_j$
  (with $W_j$ trivial for $j=1, \dots, d$) is commutative.   The difficulty in writing down examples is that there is no
  explicit way to write down such operator tuples ${\mathbb G}$ so that the associated isometric-tuple $\underline{V}$
  is commutative.

  As we have seen in Section \ref{S:BCL},
  given such a collection of  operators forming a BCL-tuple as in Definition \ref{D:BCLtuple} (with ${\mathbb W}$
  taken to be trivial for simplicity),
  the operators $V_j = M_{G_{j1}^* + z G_{j2}}$ ($j=1, \dots, d$) form an isometric tuple but there are no explicit
  criteria for deciding when it is the case that this is a commutative isometric tuple, unless $d=1,2$.

  Similarly from Theorem  \ref{Thm:reverse}, to construct examples of commutative contractive tuples,  it suffices to
  construct examples of admissible triples.  At its core, according to Definition \ref{D:admissible},
  an admissible triple consists of a pure contractive operator function $( \cD, \cD_*, \Theta)$ together with a
  collection of operators ${\mathbb G}   = \{ G_{j1}, G_{j2} \colon 1 \le j \le d\}$ on $\cD_*$, and
  a commutative unitary tuple ${\mathbb W}  = \{ W_j \colon 1 \le j \le d \}$ acting on
  $\overline{\Delta_\Theta L^2(\cD)}$ satisfying auxiliary conditions (1)-(5).
  While conditions (1) and (2) are not so difficult to analyze,   the joint-invariance property in condition (3)
  and the joint-commutativity property in condition (4)  are mysterious:  for a general $\Theta$ there is no apparent way
  to write down interesting explicit examples of potential admissible triples $({\mathbb G}, {\mathbb W}, \Theta)$
  which satisfy these additional properties, even for the $d=2$ case.  In case $\Theta$ is inner, the ${\mathbb W}$-component becomes trivial,  condition (1) is just the requirement that
  the operator pencil $G_j(z) = G_{j1}^*+ z G_{j2}$ have $H^\infty$-norm at most $1$ but
  one is still left with the nontrivial requirement (3) that $M_{G_{j1}^* + z  G_{j2}}$ leave the subspace
  $M_\Theta H^2(\cD)$ invariant.
  Unlike the case for the BCL-model for commutative isometric tuples, this flaw happens even in the $d=2$ case.

  An example which may be tractable is the case where the commutative contractive tuple $\underline{T}
  = (T_1, \dots, T_d)$ acts on a finite-dimensional Hilbert space $\cX$ and has a basis of joint eigenvectors.
  Similar examples are discussed in \cite{A-M-Dist-Var,  BDF3}.

  More detail on all these issues will appear in forthcoming work of the authors \cite{BS-Memoir}.
   \end{remark}

\noindent{\bf{Acknowledgement.}} This work was done when the second named author was visiting
Virginia Tech as an SERB Indo-US Postdoctoral Research Fellow. He wishes to thank
the Department of Mathematics, Virginia Tech for all the facilities provided to him.

\end{document}